\documentclass[reqno, 10pt, letterpaper]{amsart}
\usepackage{amsfonts,amsmath, amsthm, amssymb, latexsym, epsfig}
\usepackage[all]{xy}
\usepackage{xspace}
\usepackage{comment}
\usepackage{setspace}
\usepackage{enumerate}
\usepackage[mathscr]{eucal}

\usepackage{bbold}

	\advance \topmargin by -\headheight
	\advance \topmargin by -\headsep

	\evensidemargin \oddsidemargin
	\newcommand{\catc}{\mathfrak{C}^{*}\text{-}\mathfrak{alg}}
	\newcommand{\ftn}[3]{ #1 \colon #2 \to #3 }
	\newcommand{\setof}[2]{\ensuremath{\left\{ #1 \: \mid \: #2 \right\}}}
	\newcommand{\tc}[1]{\overline{{\tt t}#1}}
	\newcommand{\Her}{\mathcal{H}}

	\newcommand{\kk}{\ensuremath{\mathit{KK}}\xspace}
	\newcommand{\multialg}[1]{\mathcal{M}(#1)\xspace}
	\newcommand{\corona}[1]{\mathcal{Q}(#1)\xspace}
	\newcommand{\Z}{\ensuremath{\mathbb{Z}}\xspace}
	\newcommand{\C}{\ensuremath{\mathbb{C}}\xspace}

	\newcommand{\N}{\ensuremath{\mathbb{N}}\xspace}
	
	\newcommand{\K}{\ensuremath{\mathbb{K}}\xspace}
	\newcommand{\id}{\ensuremath{\operatorname{id}}}
	
	\newcommand{\ouri}{{\mathbb I}}	
	\newcommand{\Prim}{\operatorname{Prim}}
	\newcommand{\primT}{\Prim^\tau}	
	\newcommand{\primTS}{\Prim^{\tau,\Sigma}}

	\theoremstyle{plain}
	\newtheorem{thm}{Theorem}[section]
	\newtheorem{lemma}[thm]{Lemma}
	\newtheorem{theor}[thm]{Theorem}
	\newtheorem{propo}[thm]{Proposition}
	\newtheorem{corol}[thm]{Corollary}

	\theoremstyle{definition}
	\newtheorem{defin}[thm]{Definition}
	
	\newtheorem{remar}[thm]{Remark}
	
	\newtheorem{examp}[thm]{Example}

	\numberwithin{equation}{section}
	\numberwithin{figure}{section}

\newcommand{\FK}{FK}

\begin{document}
	\title{Amplified graph $C^{*}$-algebras}
	\author{S{\o}ren Eilers}
        \address{Department of Mathematical Sciences \\
        University of Copenhagen\\
        Universitetsparken~5 \\
        DK-2100 Copenhagen, Denmark}

              \email{eilers@math.ku.dk }

	\author{Efren Ruiz}
        \address{Department of Mathematics\\University of Hawaii,
Hilo\\200 W. Kawili St.\\
Hilo, Hawaii\\
96720-4091 USA}
        \email{ruize@hawaii.edu}

        \author{Adam P.W.~S{\o}rensen}
        \address{Department of Mathematical Sciences \\
        University of Copenhagen\\
        Universitetsparken~5 \\
        DK-2100 Copenhagen, Denmark}

        \email{apws@math.ku.dk}

        \date{\today}
	


        \begin{abstract}
        We provide a complete invariant for graph $C^{*}$-algebras which are amplified in
the sense that whenever there is an edge between two vertices, there are infinitely
many. The invariant used is the standard primitive ideal space adorned with a map
into $\{-1,0,1,2,\dots\}$, and we prove that the classification result is strong in
the sense that isomorphisms at the level of the invariant always lift. We
extend the classification result to cover more graphs, and give a range result for
the invariant (in the vein of Effros-Handelman-Shen) which is further used to prove
that extensions of graph $C^*$-algebras associated to amplified graphs are again
graph $C^*$-algebras of amplified graph.
        \end{abstract}

        \maketitle
        
\section{Introduction}

When classifying $C^*$-algebras we usually consider some subcategory, $\mathcal{C}$
say, of all $C^*$-algebras. 
We then hope to find a functor $\mathcal{F}$ from $\mathcal{C}$ or from the category
of all $C^*$-algebras to some other category, $\mathcal{D}$ say, with the property
that 
\[
        C_1 \cong C_2 \iff \mathcal{F}(C_1) \cong \mathcal{F}(C_2).
\]
Hopefully it is easy to determine if two objects in $\mathcal{D}$ are isomorphic.
If one somehow comes to be in possession of such a classifying functor, there are several natural questions to ask. 
Is our functor a strong classifying functor?
That is, given some isomorphism $\phi \colon \mathcal{F}(C_1) \to \mathcal{F}(C_2)$ can we find an isomorphism $\psi \colon C_1 \to C_2$ such that $\mathcal{F}(\psi) = \phi$?
What is the range of $\mathcal{F}$?
If the domain of $\mathcal{F}$ is all $C^*$-algebras, then we can ask under what conditions $\mathcal{F}(\mathfrak{A}) \in \mathcal{F}(\mathcal{C})$ guarantees that $\mathfrak{A} \in \mathcal{C}$.  

There are many examples of such functors. 
The best known is perhaps the one that sends a $C^*$-algebra $\mathfrak{A}$ to its primitive ideal space, denoted $\operatorname{Prim}(\mathfrak{A})$. 
Restricting to the category of commutative $C^*$-algebras we obtain a classifying functor, since in this case we may apply Gelfand duality, and if one restricts further to unital commutative $C^*$-algebras then this is even a strong classifying functor. 
It is also well known that the range of the functor is all locally compact Hausdorff spaces, but apart from the obvious fact that when $\operatorname{Prim} (\mathfrak{A})$ is not Hausdorff, then $\mathfrak{A}$ is not commutative, there is no obvious way to recognize the commutative $C^*$-algebras by looking at primitive ideal spaces.

Among other well known classifying functors, we have the ordered $K_0$ (with unit or scale) which classifies AF-algebras, and the graded $K_0\oplus K_1$, which classifies purely infinite simple $C^*$-algebras which are nuclear and fall in the UCT class.
These are strong classification results, the ranges of the invariants are known, and in the former case quite a lot is known about how to recognize the classified $C^*$-algebra in larger classes by looking at their $K$-theory.

In \cite{ERRlinear} it was boldly conjectured that the ideal related $K$-theory $\FK(-)$, is a classifying (up to stable isomorphism) functor for graph $C^*$-algebras with finitely many ideals. 
Supporting evidence for this conjecture can, for instance, be found in \cite{semt_classgraphalg, ERRlinear, segrer:ccfis}, where graph $C^*$-algebras with small or otherwise special ideal lattices are classified using $\FK(-)$.
Although we are apparantly still a long way away from resolving whether $\FK(-)$ is a classifying (up to stable isomorphism) functor for graph $C^*$-algebras, the conjecture raises the following additional questions:
\begin{enumerate}[A]
        \item Is $\FK(-)$ a classifying functor?
                \begin{enumerate}[1]
                        \item Is it a strong classifying functor?
                        \item Can we achieve exact classification? 
                        \item What is the range of $\FK$?
                \end{enumerate}
        \item Which relation on graphs is induced by stable isomorphism?
        \item Are the graph algebras recognizable within larger classes of C*-algebras by $\FK(-)$?
				\begin{enumerate}[1]
                         \item Is it possible to achieve permanence results for extensions of graph algebras?
                \end{enumerate}
\end{enumerate}
These questions and their answers affect one another.
We have tried to capture these connections in the following diagram.

\begin{center}
\includegraphics[width=7cm]{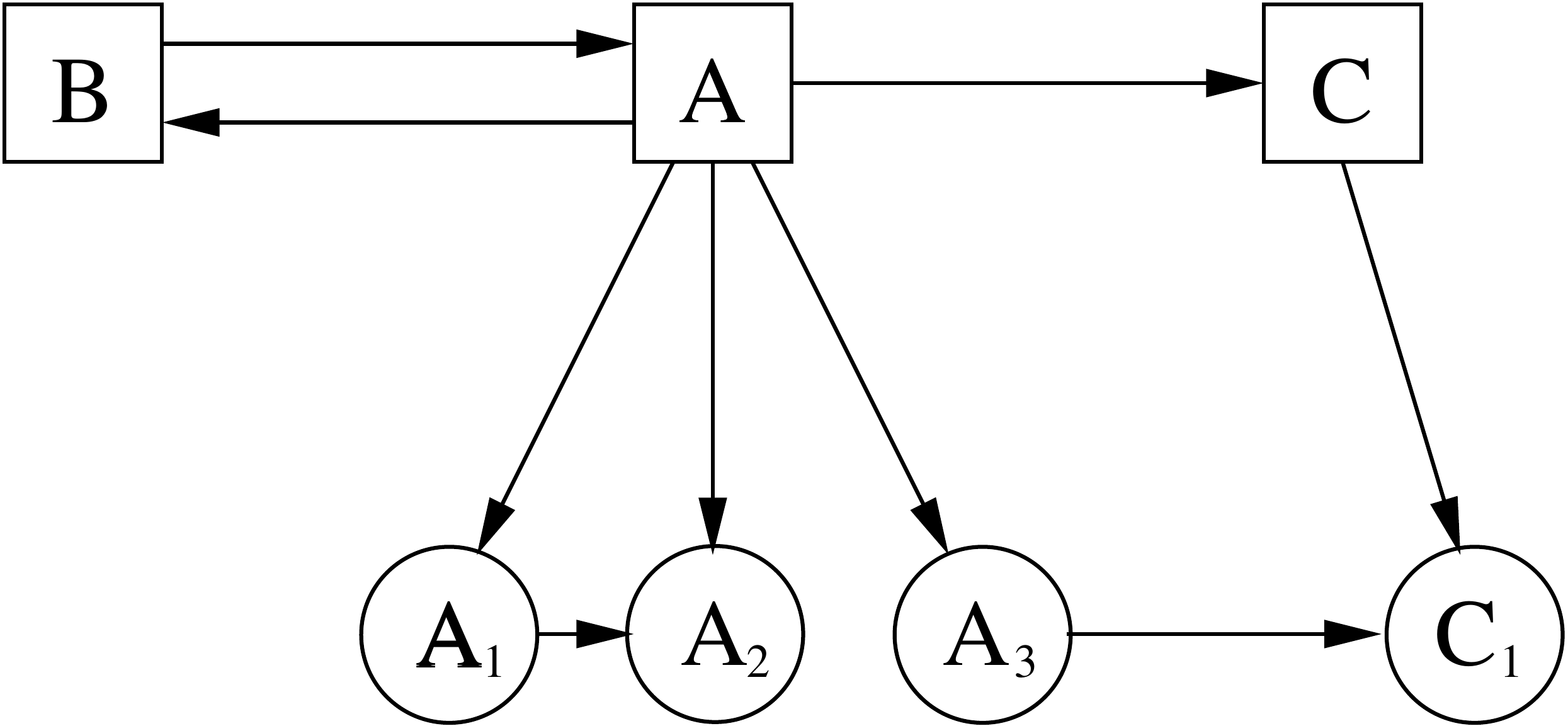}
\end{center}

In the class of graph $C^*$-algebras, results pertaining to $A$ have been given in \cite{semt_classgraphalg}, \cite{segrer:ccfis}, and \cite{ERRlinear}, and more will appear in \cite{segrer:scecc}. 
The issue $A_1$ is the subject of \cite{segr:rccconi} as well as \cite{grer:rccconiII}, \cite{segrer:scecc}, and all of these papers along with \cite{segrer:okfe} adress issue $A_2$.
$A_3$ is the subject of \cite{EKTW} as well as forthcoming work in \cite{ABK}.
The question $B$ is resolved for simple unital graph $C^*$-algebras in \cite{apwsalafranks}.
And results of relevance for $C$ and $C_1$ have been obtained in \cite{EKTW}, \cite{segrer:okfe} and \cite{ABK}.

In the present paper we resolve all the questions $A$, $B$ and $C$ for a special class of graph algebras -- imposing this time, indirectly, a requirement on the involved $K$-groups instead of on the ideal lattice.  
Instead of working directly with $\FK(-)$, we introduce, for any $C^*$-algebra, the {\bf tempered primitive ideal space}.
It turns out to be a complete invariant for the algebras we wish to study.  
This invariant, which we denote by $\primT(\mathfrak{A})$, consists of the standard primitive ideal space $\Prim(\mathfrak{A})$ along with a map $\tau:\Prim(\mathfrak{A})\to \Z\cup\{-\infty,\infty\}$ which describes the nature of the $K_0$-groups of certain sub-quotients of $\mathfrak{A}$.
A formal definition is given in Section \ref{sec:tempideal}. 

In terms of graphs, what we want to consider are graphs with finitely many vertices and the special property that if there is an edge between two vertices, then there are infinitely many edges between them.  
We call such graphs amplified. 
Naturally any vertex in such a graph is either an infinite emitter or a sink, and so the $K_0$-group is easily computed.
It is simply the free abelian group with as many generators as our graph has vertices.
Furthermore such graphs always satisfies the technical condition $(K)$ so the ideal structure of a graph algebra is readily understood from the path structure of the graph.  

An important concept for us is the \emph{transitive closure} of a graph $G$, defined in the case of $G$ by adding an edge $e$ with $s(e)=v$ and $r(e)=w$ to the graph if no such edge exists, but there is a path from $v$ to $w$ in $G$.
We denote this graph by ${\tt t} G$. 
We also need the \emph{amplification} of a graph $G$, defined by adding countably infinite number of edges from $v$ to $w$ if there exists an edge in $G$ from $v$ to $w$.
We denote this graph by $\overline{G}$. 

We can now state one of the key results of our paper. 

\begin{theor}
Let $G$ and $F$ be graphs with finitely many vertices. 
The following are equivalent.
\begin{itemize}
\item[(i)] $\tc{G} \cong \tc{F}$.

\item[(ii)] $C^{*} ( \tc{G} ) \cong C^{*} ( \tc{F} )$.

\item[(iii)] $C^{*} ( \overline{G} ) \cong C^{*} (\overline{F} )$.

\item[(iv)] $C^{*} ( \overline{G} ) \otimes \K \cong C^{*} ( \overline{F} ) \otimes \K $.

\item[(v)] $\primT( C^{*} ( \overline{G} ) ) \cong \primT( C^{*} ( \overline{F} )  )$.

\item[(vi)] $\FK( C^{*} ( \overline{G} ) ) \cong \FK( C^{*} ( \overline{F} )  )$.
\end{itemize}
\end{theor}

Aside from showing that $\FK(-)$ is a classifying functor for amplified graph algebras, the result also gives a concrete geometric description of when two amplified graphs gives rise to isomorphic algebras.  

We provide various extensions of this result; we extend to cover classification of (some) graphs where all vertices are singularities and prove an Effros-Handelman-Shen type theorem for the range of the invariant.
This is then used to show permanence properties for certain graph $C^{*}$-algebras, hence answering questions $C$ and $C_1$. 

\section{Graphs and their algebras}

Across the literature on graph algebras there is some inconsistency about how to turn the arrows when defining the graph algebra. 
We go with the notation from, for instance, \cite{ddmt_arbgraph}. 

\begin{defin}
Let $G = (G^0,G^1,s_{G},r_{G})$ be a graph. A Cuntz-Krieger $G$-family is a set of mutually orthogonal projections $\{ p_v \mid v \in G^0 \}$ and a set $\{ s_e \mid e \in G^1 \}$ of partial isometries satisfying the following conditions:
\begin{itemize}
	\item[(CK0)] $s_e^* s_f = 0$ if $e,f \in G^1$ and $e \neq f$,
	\item[(CK1)] $s_e^* s_e = p_{r_{G}(e)}$ for all $e \in G^1$,
	\item[(CK2)] $s_e s_e^* \leq p_{s_{G}(e)}$ for all $e \in G^1$, and,
	\item[(CK3)] $p_v = \sum_{e \in s_{G}^{-1}(v)} s_e s_e^*$ for all $v \in G^0$ with $0 < |s_{G}^{-1}(v)| < \infty$.
\end{itemize}
The graph algebra $C^*(G)$ is defined as the universal $C^*$-algebra given by these generators and relations.
\end{defin}

We now define a few graph concepts. 

\begin{defin}
Let $G$ be graph and let $u,v$ be vertices in $G$. 
We write $u \geq v$ if there is a path from $u$ to $v$ in $G$ or if $u = v$.
\end{defin}

\begin{defin}
Let $G$ be a graph.
A subset $H \subseteq G^0$ is called \emph{hereditary} if for all $u \in H$ we have 
\[
	u \geq v \implies v \in H.
\]
We denote by $\mathcal{H}(G)$ the lattice of hereditary sets in $G^0$. 
\end{defin}

Of particular importance to us is the amplification of a graph.

\begin{defin}
Let $G$ be a graph.
The \emph{amplification of $G$}, denoted by $\overline{G}$, is defined by $\overline{G}^{0} = G^{0}$, 
\begin{align*}
\overline{G}^{1} = \setof{ e(v,w)^{n} }{ \text{$n \in \N$, $v, w \in G^{0}$ and there exists an edge from $v$ to $w$} }, 
\end{align*}
and $s_{ \overline{G} } ( e(v,w)^{n} ) = v$, and $r_{\overline{G}} ( e(v,w)^{n} ) = w$. 

If $E = \overline{G}$ for some graph $G$ we say that $E$ is an \emph{amplified graph}.
\end{defin}

The ideal structure of graph algebras is well understood.
For amplified graphs it is especially nice.  
Since all sets of vertices automatically are saturated and the graphs always satisfy the technical condition (K), we have an isomorphism between $\mathcal{H}(\overline{G})$ and the ideals of $C^*(\overline{G})$. 
See \cite{bhrs:idealstructure} for details. 
This connection between the path structure of $\overline{G}$ and the ideals of the associated algebra motivates our next definition. 

\begin{defin}
Let $G = ( G^{0}, G^{1} , r_{G} , s_{G} )$ be a graph.
Define ${\tt t}G$ as follows:
\begin{align*}
	{\tt t}G^{0} &= G^{0}, \\
	{\tt t}G^{1} &= G^1 \cup \setof{ e(v, w ) }{ \text{there is a path but no edge from } v \text{ to } w },
\end{align*}
with range and source maps that extend those of $G$ and satisfies
\begin{align*}
s_{ {\tt t}G} ( e(v,w) ) &= v,   \\
r_{ {\tt t}G } ( e(v,w) ) &= w.
\end{align*}
\end{defin}

The idea is that in ${\tt t}G$ the relations ``there is a path between'' and ``there is an edge between'' becomes the same. 
Of course adding one edge is fairly arbitrary. 
A more natural choice would perhaps be to add as many edges from $u$ to $v$ as there paths from $u$ to $v$. 
But with that choice ${\tt t}{\tt t}G$ might not be the same as ${\tt t}G$.
We will almost only use the transitive closure together with the amplification, thus this choice is irrelevant. 

There is one final class of graphs that will be important for us.

\begin{defin}
A graph $G$ is called \emph{singular} if every vertex in $G$ is either an infinite emitter or a sink. 
\end{defin}

\begin{remar}
An amplified graph is obviously singular. 
And there are non-amplified singular graphs, cf.\ Example \ref{nonsing}.
\end{remar}

\section{A move on graphs} \label{sec:move}

In this section we will describe a simple way to alter an amplified graph without changing the isomorphism class of the associated algebra.

\begin{lemma} \label{addEdgesFinite}
Let $G$ be a graph, and let $u \in G^0$ be some vertex that emits infinitely many edges to some finite emitter $v$ in $G^0$. Let $E$ be the graph with vertex set $G^0$, edge set
 \[
	E^1 = G^1 \cup \{ f^n \mid n \in \N, f \in s_{G}^{-1}(v) \},
\]
and range and source maps that extend those of $G$ and have $r_{E}(f^n) = r_{G}(f)$ and $s_{E}(f^n) = u$. Then  $C^*(G) \cong C^*(E)$.
 \end{lemma}

In the lemma if $G$ looks like:
\[
	\xymatrix{ 
 	 	& & \bullet \\
          	 	u \ar[r]^{\infty} & v \ar@<-0.1em>[ur]\ar@<0.1em>[ur] \ar[dr] & \\
 	 	& & \bullet \\ 
	}
\]
then $E$ will look like:
\[
	\xymatrix{ 
 	 	& & \bullet \\
          	 	u \ar[r]^{\infty} \ar@/^1em/[urr]^{\infty} \ar@/_1em/[drr]_{\infty} & v \ar@<-0.1em>[ur]\ar@<0.1em>[ur] \ar[dr] & \\
 	 	& & \bullet \\ 
	}
\]

\begin{proof}
Let $\{p_v, s_e \mid v \in G^0, e \in G^1\}$ be the Cuntz-Krieger $G$-family generating $C^*(G)$.
Let 
\[
	\{ e_1, e_2, \ldots \} = \{ e \in G^1 \mid s_{G}(e) = u, r_{G}(e) = v \}.
\]
For each edge $e \in G^1 \setminus \{e_n \mid n \in \N\}$ we let $t_e = s_e$.
For each $e_n$ we let $t_{e_n} = s_{e_{2n-1}}$. 
For each  $f^n$ we let $t_{f^n} = s_{e_{2n}}s_f$. 

We claim that $\{ p_v, t_e \mid v \in E^0, e \in E^1 \}$ is a Cuntz-Krieger $E$-family in $C^*(G)$. 
First we check that all the $t_e$, $e \in E^1$, have orthogonal ranges.
Let $e,f \in G^1$. If neither $e$ nor $f$ has source $u$ we have $t_e = s_e$ and $t_f = s_f$, so they have orthogonal ranges.
Suppose now that $s_{G}(e) = u$ and $s_{G}(f) \neq u$.
Then we can write $t_e = s_g x$ where $g$ is an edge with $s_{G}(g) = u$ and $x$ is some element in $C^*(G)$.
Thus we have 
\[
	t_e^* t_f = x^* s_g^* s_f = 0,
\]
the last equility holds since $s_{G}(f) \neq u$ so $g \neq f$.
From this we also get $t_f^* t_e = (t_e^* t_f)^* = 0$.
We now consider the case where both $e$ and $f$ have source $u$.
The only case which is different from before, is if $e = g^n$ and $f = h^m$ for some edges $h,g \in s_{G}^{-1}(v)$ and $n,m \in  \N$.
In this case we have 
\begin{align*}
	t_e^* t_f 	&=  s_g^* s_{e_ {2n}}^* s_{e_{2m}} s_h = \delta_{n,m} s_g^* p_v s_f^* \\
			&= \delta_{n,m} s_g^* s_f = \delta_{n,m} \delta_{g,f} p_{r_{G}(g)} \\
			&= \delta_{n,m} \delta_{g,f} p_{r_{E}(e)} ,
\end{align*}
which gives the desired result.

Note that by the above computations the relation $t_e^* t_e = p_{r_{E}(e)}$ holds for all $e \in E^1$.
The sum-relation $(CK3)$ holds at all vertices, since the only vertex, where we changed the out going edges (and the corresponding partial isometries) is an infinite emitter. 
So it only remains to verify that $t_e t_e^* \leq p_{s_{E}(e)}$ for all $e \in E^1$.
This is easily seen to be true unless $e = f^n$ for some $f \in s_{G}^{-1}(v)$ and $n \in \N$.
But even in this case we see that 
\[
	t_e t_e^* = s_{e_{2n}} s_f s_f^* s_{e_{2n}}^* \leq s_{e_{2n}} s_{e_{2n}}^* \leq p_{s_{G}(e_{2n})} = p_u = p_{s_{E}(e)}.
\]	
Hence $\{p_v, t_e\}$ is a Cuntz-Krieger $E$-family.

By universality we get a $*$-homomorphism $\phi \colon C^*(E) \to C^*(G)$.
We claim that $\phi$ is an isomorphism. 
The only generators of $C^*(G)$ that are not trivially in $\phi(C^*(E))$ are $s_{e_{2n}}$, $n  \in \N$. 
To see that these generators are in the image we fix some $n \in \N$. 
For each $f \in s_{G}^{-1}(v)$ we have 
\[
	t_{f^n} t_{f}^* = s_{e_{2n}} s_f s_f^*.
\]
Since $v$ is a finite emitter, we get
\[
	\sum_{f \in s_{G}^{-1}(v)} t_{f^n} t_{f}^* = s_{e_{2n}} \left( \sum_{f \in s_{G}^{-1}(v)} s_f s_f^* \right) = s_{e_{2n}} p_{v} = s_{e_{2n}}.
\]
So $e_{s_{2n}}$ is in the image of $\phi$, and hence $\phi$ is surjective.

We now turn to injectivity. 
We will define a strongly continuous action $\alpha$ of $\mathbb{T}$ on $C^*(G)$. 
Let $H = G^0 \setminus \{ e_{2n} \mid n \in \N \}$.
For each $z \in \mathbb{T}$ define
\[
	\alpha_z(s_e) = \begin{cases} z s_e, & \text{ if } e \in H \\ s_e, & \text{ if } e \notin H \end{cases},
\]
and 
\[
	\alpha_z(p_v) = p_v, \quad v \in G^0.
\]
By universality this defines an action of $\mathbb{T}$ on $C^*(G)$. 
A standard argument will show that $\alpha$ is strongly continuous.
If $\gamma$ is the standard gauge action on $C^*(E)$ then we have 
\[
	\phi \circ \gamma_z = \alpha_z \circ \phi, \quad \text{ for all $z \in \mathbb{T}$}.
\]
The gauge-invariant uniqueness theorem \cite[Theorem 2.1]{bhrs:idealstructure} now implies that $\phi$ is injective. 
\end{proof}

We do not need to add edges from $u$ to all the the vertices $v$ emits to. 

\begin{corol} \label{finiteAddOnlySome}
Let $G$ be a graph, $u \in G^0$ an infinite emitter.
Fix a finite emitter $v$ that $u$ emits infinitely to, and fix an edge $f \in s_{G}^{-1}(v)$.
Let $E$ be the graph with vertex set $G^0$, edge set
 \[
	E^1 = G^1 \cup \{ f^n \mid n \in \N \},
\]
and range and source maps that extend those of $G$ and have $r_{E}(f^n) = r_{G}(f)$ and $s_{E}(f^n) = u$.
Then $C^*(E) \cong C^*(F)$.
\end{corol}
\begin{proof}
Applying Lemma \ref{addEdgesFinite} to both $E$ and $G$ yields isomorphic graphs. 
\end{proof}

The requirement that $v$ is a finite emitter seems somewhat artificial, as we are focusing on  a single edge leaving $v$.
We now wish to remove that requirement. 
To do that, we use the out-splittings of Bates and Pask \cite{tbdp:flow}.
For the convenience of the reader, we record a special case of \cite[Theorem 3.2]{tbdp:flow}:

\begin{theor}[Out-split] \label{outSplit}
Let $G$ be a graph, and let $v \in G^0$.
Given an non-empty subset $\mathcal{E}_1$ of $s_{G}^{-1}(v)$ such that $\mathcal{E}_0 = s_{G}^{-1}(v) \setminus \mathcal{E}_1$ is non-empty, we define a graph $G_{os} = (G_{os}^0, G_ {os}^1, r_{os}, s_{os})$ by 
\begin{align*}
 	G_{os}^0 & = (G^0 \setminus \{ v \}) \cup \{ v^0, v^1 \}, \\
	G_{os}^1 & = \left( G^1 \setminus r_{G}^{-1}(v) \right) \cup \{ e^0, e^1 \mid e \in E^1, r_{G}(e) = v \}.
\end{align*}
For $e \notin r_{G}^{-1}(v)$ we let $r_{os}(e) = r_{G}(e)$, for $e \in r_{G}^{-1}(v)$ we let $r_{os}(e^i) = v^i$, $i = 0,1$.
For $e \notin s_{G}^{-1}(v)$ we let $s_{os}(v) = s_{G}(e)$, for $e \in s_{G}^{-1}(v)$ we let $s_{os}(e) = v^i$ if $e \in \mathcal{E}_i$, $i= 1,2$.

If $\mathcal{E}_1$ is finite then $C^*(G)\simeq C^*(G_{os})$.
\end{theor}

We also explicitly write down how to go back. 
Note that these two theorems are two ways of saying the same thing.

\begin{theor}[Out-amalgamation] \label{outAmalgamate}
Let $G$ be a graph, and let $v^0, v^1 \in G^0$. 
We now define a new graph $G_{oa} = (G_{oa}^0, G_{oa}^1, r_{oa}, s_{oa})$ by
\begin{align*}
	G_{oa}^0 & = (G^0 \setminus \{ v^0, v^1 \}) \cup \{ v \}, \\
	G_{oa}^1 & = G^1 \setminus r_{G}^{-1}(v^1).
\end{align*}
For $e \in r_{G}^{-1}(v^0)$ we let $r_{oa}(e) = v$, for $e \notin r_{G}^{-1}(v^0)$ we let $r_{oa}(e) = r_{G}(e)$.
For $e \in s_{G}^{-1}(v^i)$ we let $s_{os}(e) = v$, for $e \notin s_{G}^{-1}(v^i)$ we let $s_{oa}(e) = s_{G}(e)$, $i=1,2$.

If $v_1$  is a finite emitter then $C^*(G) \cong C^*(G_{oa})$.
\end{theor}

\begin{remar}
Out-splitting and out-amalgamating are inverse operations.
More specifically, if we first out-split according to some partition of $s^{-1}(v)$, and then out-amalgamate $v^0$ and $v^1$ we get back to where we started. 
\end{remar}

\begin{lemma} \label{addOnlySome}
Let $G$ be a graph, $u \in G^0$ an infinite emitter, and $v$ a vertex that $u$ emits infinitely to. Fix an edge $f \in s_{G}^{-1}(v)$.
Let $E$ be the graph with vertex set $G^0$, edge set
 \[
	E^1 = G^1 \cup \{ f^n \mid n \in \N \},
\]
and range and source maps that extend those of $G$ and have $r_{E}(f^n) = r_{G}(f)$ and $s_{E}(f^n) = u$.
Then $C^*(E) \cong C^*(F)$.
\end{lemma}

\begin{proof}
There is nothing to prove if $f$ is a loop, since then $E \cong G$. 
In the case $v$ is a finite emitter, we can just appeal to Corollary \ref{finiteAddOnlySome}.

Let us consider the case where $v$ is an infinite emitter and $r_G(f) \neq v$.
Define $\mathcal{E}_1 = \{ f \}$ and $\mathcal{E}_0 = s_{G}^{-1}(v) \setminus \{ f \}$. 
We out-split according to that partition $\mathcal{E}_0, \mathcal{E}_1$ of $s_{G}^{-1}(v)$, to obtain a graph $G_{os}$ that has vertex set $(G^0 \setminus \{ v \}) \cup \{ v^0, v^1 \}$.
By Theorem \ref{outSplit} $C^*(G) \cong C^*(G_{os})$.
Using Corollary \ref{finiteAddOnlySome} on $u$ and $v^1$ yields a graph $F$ with $C^*(G_{os}) \cong C^*(F)$.

We now out-amalgamate $v^0, v^1$ in $F$ to get a graph $F_{oa}$.
By Theorem \ref{outAmalgamate} $C^*(F) \cong C^*({F_{oa}})$ since $v^1$ is a finite emitter. 
We claim that $F_{oa} \cong E$. 
They both have the same vertex set as $G$. 
Given two vertices $x,y$ such that $x \neq u$, there are the same number of edges from $x$ to $y$ in both $E$ and $F_{oa}$, as in both cases there is the same number of edges from $x$ to $y$ as there is in $G$.
For any vertex $y$ other than $r_G(f)$ we must have that there are the same number of edges from $u$ to $y$ in both $F_{oa}$ and $E$ as there is in $G$.
But if $y = r_G(f)$ then $u$ has infinitely many edges to $y$ in both $E$ and $F_{oa}$. 
Thus $F_{oa} \cong E$.

In conclusion:
\[
	C^*(G) \cong C^*(G_{os}) \cong C^*(F) \cong C^*(F_{oa}) \cong C^*(E). \qedhere
\] 
\end{proof}

The next example illustrates the different graphs used in the above proof.
\begin{examp}
Let $G$ be the graph:
\[
	\xymatrix{ 
 	 	& & \bullet \\
          	 	u \ar[r]_{\infty} & v \ar[ur]_{\infty} \ar[dr] & \\
 	 	& & \bullet \\ 
	}
\]
The first step in the proof is to out-split $G$ at $v$. 
This results in the graph $G_{os}$:
\[
	\xymatrix{ 
 	 	& v^1 \ar[r]_{f}& \bullet \\
          	 	u \ar[r]_{\infty} \ar[ur]_{\infty} & v^0 \ar[ur]_{\infty} \ar[dr] & \\
 	 	& & \bullet \\ 
	}
\]
Then Corollary \ref{finiteAddOnlySome} is applied, yielding $F$:
\[
	\xymatrix{ 
 	 	& v^1 \ar[r]_{f}& \bullet \\
          	 	u \ar[r]_{\infty} \ar[ur]_{\infty}  \ar@/^3em/[urr]^{\infty} & v^0 \ar[ur]_{\infty} \ar[dr] & \\
 	 	& & \bullet \\ 
	}
\]
Which we finally out-amalgamate to $F_{oa}$:
\[
	\xymatrix{ 
 	 	& & \bullet \\
          	 	u \ar[r]_{\infty} \ar@/^3em/[urr]^{\infty} & v \ar[ur]_{\infty} \ar[dr] & \\
 	 	& & \bullet \\ 
	}
\]
We see that $F_{oa} \cong E$. 
\end{examp} 

We now can now present the final version of our move. 

\begin{theor} \label{moveT}
Let $\alpha = \alpha_1 \alpha _2 \cdots \alpha_n$ be a path in a graph $G$. 
Let $E$ be the graph with vertex set $G^0$, edge set
 \[
	E^1 = G^1 \cup \{ \alpha^m \mid m \in \N \},
\]
and range and source maps that extend those of $G$ and have $r_{E}(\alpha^m) = r_{G}(\alpha)$ and $s_{E}(\alpha^m) = s_{G}(\alpha)$.
If 
\[
	|s_{G}^{-1}(s_{G}(\alpha_1) ) \cap r_{G}^{-1}(r_{G}(\alpha_1))| = \infty,
\]
then $C^*(G) \cong C^*(E)$.
\end{theor}

\begin{proof}
Apply Lemma \ref{addOnlySome} a number of times (first adding edges, and then taking away the unwanted ones).
\end{proof}

\begin{corol} \label{transitiveClosure}
If $G$ is a graph with $|G^0|<\infty$, then $C^*(\overline{G}) \cong C^*(\overline{{\tt t}G})$.
\end{corol}

\begin{proof}
For any path $\alpha$ in an amplified graph, we have 
\[
	|s_{G}^{-1}(s_{G}(\alpha_1) ) \cap r_{G}^{-1}(r_{G}(\alpha_1)) | = \infty,
\]
so Theorem \ref{moveT} proves the desired result.
\end{proof}

We may ask ourselves: How important is the requirement 
\[
	| s_{G}^{-1}(s_{G}(\alpha_1) ) \cap r_{G}^{-1}(r_{G}(\alpha_1)) | = \infty,
\]
in Theorem \ref{moveT};
in particular, can we replace it by simply requiring that there are infinitely many paths from $s_{G}(\alpha)$ to $r_{G}(\alpha)$?
It turns out we cannot.
Witness:

\begin{examp}\label{nonsing}
Consider the graph $G$:
\[
	\xymatrix{
 		x & \ar[l]_{\infty} u \ar[r] & y \ar[r]^{\infty} & z
	}.
\]	
There are infinitely many paths from $u$ to $z$, so one might hope that $C^*(G)$ is isomorphic to the algebra of the graph  $E$:
\[
	\xymatrix{
 		x & \ar[l]_{\infty} u \ar@/_1em/[rr]_{\infty} \ar[r] & y \ar[r]^{\infty} & z
	}.
\]
However, $C^*(G)$ has seven ideals and $C^*(E)$ only has six.
\end{examp}

The problem in the graph above is that $u$ is a breaking vertex for $\{x\}$ in $G$ but not in $E$. Assuming that no vertices are breaking in this way, more general results are available, see Corollary \ref{c:singularisamp}.

\section{The tempered primitive ideal space} \label{sec:tempideal}

In this section we introduce the invariant with which we shall work. To state it in sufficient generality, we need a preliminary discussion about $C^{*}$-algebras over $X$.  Most of the facts given about $C^{*}$-algebras over $X$ are taken from \cite{rmrn:bootstrap}.

Let $X$ be a topological space and let $\mathbb{O}( X)$ be the set of open subsets of $X$, partially ordered by set inclusion $\subseteq$.  A subset $Y$ of $X$ is called \emph{locally closed} if $Y = U \setminus V$ where $U, V \in \mathbb{O} ( X )$ and $V \subseteq U$.  The set of all locally closed subsets of $X$ will be denoted by $\mathbb{LC}(X)$.  The set of all connected, non-empty locally closed subsets of $X$ will be denoted by $\mathbb{LC}(X)^{*}$.  

The partially ordered set $( \mathbb{O} ( X ) , \subseteq )$ is a \emph{complete lattice}, that is, any subset $S$ of $\mathbb{O} (X)$ has both an infimum $\bigwedge S$ and a supremum $\bigvee S$.  More precisely, for any subset $S$ of $\mathbb{O} ( X )$, 
\begin{equation*}
\bigwedge_{ U \in S } U = \left( \bigcap_{ U \in S } U \right)^{\circ} \quad \mathrm{and} \quad \bigvee_{ U \in S } U = \bigcup_{ U \in S } U.
\end{equation*}

For a $C^{*}$-algebra $\mathfrak{A}$ the set of closed ideals of $\mathfrak{A}$, partially ordered by $\subseteq$ is a complete lattice. More precisely, for any set $S$ of ideals, 
\begin{equation*}
\bigwedge_{ \mathfrak{I} \in S } \mathfrak{I} = \bigcap_{ \mathfrak{I} \in S } \mathfrak{I}  \quad \mathrm{and} \quad \bigvee_{ \mathfrak{I} \in S } \mathfrak{I} = \overline{ \sum_{ \mathfrak{I} \in S } \mathfrak{I} }.
\end{equation*}

\begin{defin}
Let $\mathfrak{A}$ be a $C^{*}$-algebra.  Let $\mathrm{Prim} ( \mathfrak{A} )$ denote the \emph{primitive ideal space} of $\mathfrak{A}$, equipped with the usual hull-kernel topology, also called the Jacobson topology.

Let $X$ be a topological space.  A \emph{$C^{*}$-algebra over $X$} is a pair $( \mathfrak{A} , \psi )$ consisting of a $C^{*}$-algebra $\mathfrak{A}$ and a continuous map $\ftn{ \psi }{ \mathrm{Prim} ( \mathfrak{A} ) }{ X }$.  A $C^{*}$-algebra over $X$, $( \mathfrak{A} , \psi )$, is \emph{separable} if $\mathfrak{A}$ is a separable $C^{*}$-algebra.  We say that $( \mathfrak{A} , \psi )$ is \emph{tight} if $\psi$ is a homeomorphism.  
\end{defin}

We always identify $\mathbb{O} ( \mathrm{Prim} ( \mathfrak{A} ) )$ with the ideals in $\mathfrak{A}$ using the lattice isomorphism
\begin{equation*}
U \mapsto \bigcap_{ \mathfrak{p} \in \mathrm{Prim} ( \mathfrak{A} ) \setminus U } \mathfrak{p}.
\end{equation*}
 Let $( \mathfrak{A} , \psi )$ be a $C^{*}$-algebra over $X$.  Then we get a map $\ftn{ \psi^{*} }{ \mathbb{O} ( X ) }{ \mathbb{O} ( \mathrm{Prim} ( \mathfrak{A} ) )}$ defined by
\begin{equation*}
U \mapsto \setof{ \mathfrak{p} \in \mathrm{Prim} ( \mathfrak{A} ) }{ \psi ( \mathfrak{p} ) \in U } = \mathfrak{A}[ U ].
\end{equation*}
For $Y = U \setminus V \in \mathbb{LC} ( X )$, set $\mathfrak{A}[Y] = \mathfrak{A} [U] / \mathfrak{A}[V]$.   By Lemma 2.15 of \cite{rmrn:bootstrap}, $\mathfrak{A} [ Y]$ does not depend on $U$ and $V$.

\begin{remar}
By Example 2.16 of \cite{rmrn:bootstrap}, if $\mathrm{Prim} ( \mathfrak{A} )$ is finite, then for every $x \in \mathrm{Prim} ( \mathfrak{A} )$, $\{ x \} \in \mathbb{LC} ( \mathrm{Prim} ( \mathfrak{A} ) )$ and $\mathfrak{A}[ \{ x \} ]$ is simple.  Moreover, every simple sub-quotient of $\mathfrak{A}$ is of the form $\mathfrak{A} [ \{ x \} ]$ for some $x \in \mathrm{Prim} ( \mathfrak{A} )$.  To shorten the notation, we set $\mathfrak{A} [ x ] = \mathfrak{A} [ \{ x \} ]$ for each $x \in \mathrm{Prim} ( \mathfrak{A} )$.
\end{remar}

\begin{defin}
Let $\mathfrak{A}$ and $\mathfrak{B}$ be $C^{*}$-algebras over $X$.  A homomorphism $\ftn{ \phi }{ \mathfrak{A} }{ \mathfrak{B} }$ is \emph{$X$-equivariant} if $\phi ( \mathfrak{A} [U] ) \subseteq \mathfrak{B} [U ]$ for all $U \in \mathbb{O}(X)$.  Hence, for every $Y = U \setminus V$, $\phi$ induces a homomorphism $\ftn{ \phi_{Y} }{ \mathfrak{A} [ Y ] }{ \mathfrak{B} [Y] }$.  Let $\catc(X)$ be the category whose objects are $C^{*}$-algebras over $X$ and whose morphisms are $X$-equivariant homomorphisms.  
\end{defin}

We can now define what we mean by ideal related $K$-theory. 
This has been known as filtrated $K$-theory in \cite[Definition 2.1]{ERRlinear}. 

\begin{defin}
Let $\mathfrak{A}$ be a $C^*$-algebra over some finite set $X$. 
Whenever we have open sets $U \subseteq V \subseteq W \subseteq X$ we have nested ideals $\mathfrak{A}[U] \triangleleft \mathfrak{A}[V] \triangleleft \mathfrak{A}[W] \triangleleft \mathfrak{A}$. 
Hence we get the following six-term sequence in $K$-theory 
\[
	\xymatrix{
		K_0(\mathfrak{A}[V \setminus U]) \ar[r]^{\iota_*} & K_0(\mathfrak{A}[W \setminus U]) \ar[r]^{\pi_*} & K_0(\mathfrak{A}[W \setminus V]) \ar[d]^{\delta} \\
		K_1(\mathfrak{A}[W \setminus V]) \ar[u]^{\delta} & K_1(\mathfrak{A}[W \setminus U]) \ar[l]^{\pi_*} & K_1(\mathfrak{A}[V \setminus U]) \ar[l]^{\iota_*}
	}
\]
The ideal related $K$-theory $\FK_X$ is the collection of all the $K$-groups (with $K_0$ ordered) that arise in this way together with all the bounding maps $\{ \iota_{*}, \pi_*, \delta \}$. 

Let $\mathfrak{B}$ be another $C^*$-algebra over $X$ and let $\psi$ be a self-homeomorphism of $X$.
We say that $\FK_X(\mathfrak{A}) \cong \FK_X(\mathfrak{B})$ if there exists group isomorphism $\alpha_*^{U,V} \colon K_*(\mathfrak{A}[U \setminus V]) \to K_*(\mathfrak{B}[U \setminus V])$, preserving all the bounding maps and such that $\alpha_0^{U,V}$ is an order isomorphism.

If $X = \mathrm{Prim} ( \mathfrak{A} )$, then we get all $K$-groups of all sub-quotients. 
In this case we write $\FK$ instead of $\FK_{\mathrm{Prim} ( \mathfrak{A} )}$.
\end{defin}

In the case of a simple $C^*$-algebra $\mathfrak{A}$, $\FK(\mathfrak{A})$ collapses to $(K_0(\mathfrak{A}), K_0^+(\mathfrak{A}), K_1(\mathfrak{A}))$. 
If $\mathfrak{A}$ has precisely one ideal, $\mathfrak{J}$ say, $\FK(\mathfrak{A})$ is just the six-term exact sequence in $K$-theory coming from $\mathfrak{J} \rightarrowtail \mathfrak{A} \twoheadrightarrow \mathfrak{A} / \mathfrak{J}$.

\subsection{The Alexandrov Topology} \label{sec:alexandrov}

\begin{defin}
Let \((X,\leq)\) be a preordered set.  A subset \(S\subseteq
  X\) is called \emph{Alexandrov-open} if \(S\ni x\le y\)
  implies \(y \in S\).  The Alexandrov-open subsets form a
  topology on~\(X\) called the \emph{Alexandrov topology}.
\end{defin}

A subset of~\(X\) is closed in the Alexandrov topology if and
only if \(S\ni x\) and \(x\ge y\) imply \(S\ni y\).  It is
locally closed if and only if it is \emph{convex}, that is,
\(x\le y\le z\) and \(x,z\in S\) imply \(y\in S\).  In
particular, singletons are locally closed.

The specialisation preorder for the Alexandrov topology is the
given preorder.  Moreover, a map \((X,\le)\to (Y,\le)\) is
continuous for the Alexandrov topology if and only if it is
monotone.  Thus we have identified the category of preordered
sets with monotone maps with a full subcategory of the category
of topological spaces.

It also follows that if a topological space carries an
Alexandrov topology for some preorder, then this preorder must
be the specialisation preorder.  In this case, we call the
space an \emph{Alexandrov space} or a \emph{finitely generated space}. The following theorem, \cite[Corollary 2.33]{rmrn:bootstrap}, provides an equivalent descriptions of Alexandrov spaces.

\begin{theor}
  \label{t:finite_Alexandrov}
  Any finite topological space is an Alexandrov space.  Thus
  the construction of Alexandrov topologies and specialisation
  preorders provides a bijection between preorders and
  topologies on a finite set.
\end{theor}

\begin{remar}
A useful way to represent finite partially ordered sets 
is via finite
\emph{directed acyclic graphs}.

To a partial order~\(\preceq\) on~\(X\), we associate the
finite acyclic graph with vertex set~\(X\) and with an
arrow \(x\leftarrow y\) if and only if \(x\prec y\) and there
is no \(z\in X\) with \(x\prec z\prec y\).  We can recover the
partial order from this graph by letting \(x\preceq y\) if and
only if the graph contains a directed path \(x\leftarrow
x_1\leftarrow \dotsb \leftarrow x_n\leftarrow y\).
\end{remar}

\subsection{Invariant for $C^{*}$-algebras over a finite topological space}

Let $X$ and $Y$ be topological spaces and let $\ftn{ \alpha }{ X }{ Y }$ be a homeomorphism.  Define $\ftn{ \widetilde{ \alpha } }{ \mathbb{LC} ( X ) }{ \mathbb{LC} (Y) }$ by $\widetilde{ \alpha } ( U \setminus V ) = \alpha ( U) \setminus \alpha(V)$, where $U, V \in \mathbb{O} ( X )$ with $V \subseteq U$.  Note that since $\alpha$ is a homeomorphism, $\widetilde{ \alpha }$ is a well-defined bijection.

\begin{defin}
Let $\mathfrak{A}$ be a $C^{*}$-algebra over $X$.  Then
\begin{itemize}
\item[(1)] $\ouri_{X}( \mathfrak{A} )$ is the ordered pair $( X, \tau_{\mathfrak{A} } )$, where $\ftn{ \tau_{ \mathfrak{A} } }{X }{ \Z \cup \{ -\infty, \infty \} }$ such that 
\[
\tau_{ \mathfrak{A} } ( x ) =
\begin{cases}
 - \mathrm{rank} ( \mathfrak{A} [ x ] ) &\text{if $K_{0} ( \mathfrak{A}  [ x ] )_{+} \neq K_{0} ( \mathfrak{A} [x] )$} \\
  \mathrm{rank} ( \mathfrak{A} [ x ] ) &\text{if  $K_{0} ( \mathfrak{A}  [ x ] )_{+} = K_{0} ( \mathfrak{A} [x] )$}
\end{cases}
\]

\item[(2)] $\ouri_{ X , \Sigma } ( \mathfrak{A} )$ is the ordered triple $( X , \tau_{ \mathfrak{A} } , \sigma_{ \mathfrak{A} } )$, where the ordered pair $( X , \tau_{ \mathfrak{A} } ) = \ouri_{ X } ( \mathfrak{A} )$ and $\ftn{ \sigma_{ \mathfrak{A} } }{ \mathbb{LC} ( X ) }{ \{ 0, 1 \} }$ such that 
\begin{align*}
\sigma_{ \mathfrak{A} } (Y) = 
\begin{cases}
1, &\text{if $\mathfrak{A} [ Y ]$ is unital}\\
0, &\text{otherwise}
\end{cases}
\end{align*}
\end{itemize}
If $X = \mathrm{Prim} ( \mathfrak{A} )$, we set $\ouri_{ \mathrm{Prim} ( \mathfrak{A} ) } ( \mathfrak{A} ) = \primT ( \mathfrak{A} )$ and $\ouri_{ \mathrm{Prim} ( \mathfrak{A} ) , \Sigma } ( \mathfrak{A} ) = \primT_{ \Sigma } ( \mathfrak{A} )$.

Let $\mathfrak{B}$ be a $C^{*}$-algebra over $Y$.  
\begin{itemize}
\item[(i)] We say that $\ouri_{X}( \mathfrak{A} )$ and $\ouri_{Y} ( \mathfrak{B} )$ are  \emph{isomorphic}, denoted by $\ouri_{X}( \mathfrak{A} ) \cong \ouri_{Y}( \mathfrak{B} )$, if there exists a homeomorphism $\ftn{ \alpha }{ X }{ Y }$ such that $\tau_{ \mathfrak{B} } \circ \alpha = \tau_{ \mathfrak{A} }$.

\item[(ii)] We say that $\ouri_{X, \Sigma}( \mathfrak{A} )$ and $\ouri_{Y, \Sigma} ( \mathfrak{B} )$ are \emph{isomorphic}, denoted by $\ouri_{X, \Sigma }( \mathfrak{A} ) \cong \ouri_{Y, \Sigma}( \mathfrak{B} )$, if there exists a homeomorphism $\ftn{ \alpha }{ X }{ Y }$ such that $\tau_{ \mathfrak{B} } \circ \alpha = \tau_{ \mathfrak{A} } $ and $\sigma_{ \mathfrak{B} } \circ \widetilde{\alpha} = \sigma_{\mathfrak{A}}$.
\end{itemize}

If $\mathfrak{B}$ is a $C^{*}$-algebra over $X$ and $\alpha$ in (i) or (ii) is $\id_{X}$, then we write $\ouri_{X} ( \mathfrak{A} ) \equiv \ouri_{X} ( \mathfrak{B} )$ and $\ouri_{X, \Sigma } ( \mathfrak{A} ) \equiv \ouri_{ X , \Sigma } ( \mathfrak{B} )$.
\end{defin}

\section{Classification of amplified graph $C^{*}$-algebras with finitely many vertices}\label{classamplified}

In order to use our invariant we will need to study simple sub-quotients of $C^*(\tc{G})$. 
We focus on a nice class of hereditary subsets.  

\begin{defin}
If $G$ is a graph we define a map $\iota_G \colon G^0 \to \Her(G)$ by 
\begin{align*}
	\iota_G(v) = \{ w \in G^0 \mid v \geq w \}. 
\end{align*} 
\end{defin}

By definition, we have $u \geq u$ for any vertex $u \in G^0$.
Hence $u \in \iota_G(u)$ for all vertices $u$. 
We claim that the sets $\iota_G(v)$ are ``special'' in the lattice of hereditary sets.
More specifically:
\begin{lemma}\label{desciotaimg}
The following are equivalent for a set $H \in \Her(G)$:
\begin{enumerate}[(i)]
	\item There is a $H_0 \in \Her(G)$ such that $H_0 \subsetneq H$ and for every $H_1 \in \Her(G)$ we have $H_1 \subsetneq H \Rightarrow H_1 \subseteq H_0$. I.e. $H$ contains a largest proper hereditary subset. 
	\item $H = \iota_G(u)$ for some $u \in G^0$. 
\end{enumerate}
\end{lemma}
\begin{proof}
First we show $(i) \Rightarrow (ii)$: Suppose that $u \in H \setminus H_0$. We claim that for all $v \in H$, we must have $u \geq v$. Since if $u \not \geq v$ for some $v \in H$, then $\iota_G(u) \subsetneq H$, but since $u \notin H_0$, we have $\iota_G(u) \not \subseteq H_0$. That is a contradiction, so we must have $u \geq v$ for all $v \in H$. But then $\iota_G(u) = H$.

We now show $(ii) \Rightarrow (i)$: Suppose $H = \iota_G(u)$. Put
\[
	H_0 = \bigcup_{\{ w \in H \mid  w \not \geq u\}} \iota_G(w). 
\]
Note that $H_0 \subsetneq H$. If $H_1 \subsetneq H$, then no $w \in H_1$ has a path to $u$. Thus 
\[
	H_1 = \bigcup_{w \in H_1} \iota_G(w) \subseteq \bigcup_{\{ w \in H \mid  w \not \geq u\}} \iota_G(w) = H_0. \qedhere
\] 
\end{proof}

\begin{corol} \label{iotalatticeiso}
If $\Phi \colon \Her(G) \to \Her(E)$ is a lattice isomorphism then 
\[
	\Phi(\iota_G(G^0)) = \iota_E(E^0).
\]
\end{corol}

Another important feature of the sets $\iota_G(u)$ is that we can use them to get simple sub-quotients of $C^*(\tc{G})$.

\begin{defin}
Let $G = (G^0, G^1, s_G, r_G)$ be a graph. 
Put
\[
	\langle u \rangle_G = \{ v \in G^0 \mid \iota_G(v) = \iota_G(u) \}.
\]
Let
\[
	G_u = (\langle u \rangle_G, s^{-1}_G(\langle u \rangle_G) \cap r^{-1}_{G}(\langle u \rangle_G), s_G \vert_{\langle u \rangle_G}, r_G \vert_{\langle u \rangle_G}).
\]
\end{defin}

Observe that if $H_0$ is the largest proper hereditary subset of $\iota_G(u)$, then 
\begin{align*}
	\langle u \rangle	& = \{ v \in G^0 \mid u \geq v \text{ and } v \geq u \} \\
				& = \iota_G(u) \setminus \bigcup_{ v \not \geq u} \iota_G(v) \\
				& = \iota_G(u) \setminus \bigcup_{\{v \in \iota_G(v) \mid v \not \geq u\}} \iota_G(v) \\
				& = \iota_G(u) \setminus H_0 
\end{align*}

From this we easily see that if $u \in \tc{G}^0$ then $C^*(\tc{G}_u)$ is simple. 
We also get that it is a simple sub-quotient (up to stable isomorphism).

\begin{lemma} \label{simplesubq}
Let $G$ be a graph. 
Let $u \in \tc{G}^0$ and let $H_0$ be the largest proper hereditary subset of $H = \iota_{\tc{G}}(u)$.
We have 
\[ 
	I_H / I_{H_0} \otimes \K \cong C^*(\tc{G}_u) \otimes \K.
\]
In particular $|\langle u \rangle| = \mathrm{rank} ( K_{0} ( I_H / I_{H_0}) )$. 
\end{lemma}

We are now ready to classify amplified graph algebras. 

\begin{propo}\label{p:invtrancls}
Let $G$ and $E$ be graphs with finitely many vertices. 
If $\primT(C^*(\tc{G})) \cong \primT(C^*(\tc{E}))$ then $\tc{G} \cong \tc{E}$. 
\end{propo}

\begin{proof}
Suppose $\primT(C^*(\tc{G})) \cong \primT(C^*(\tc{E}))$.
Then there exists a homeomorphism $\ftn{ \phi }{ \mathrm{Prim} ( C^{*} ( \tc{G} ) ) }{ \mathrm{Prim} ( C^{*} ( \tc{E} ) ) }$ such that 
\[
	\tau_{ C^{*} ( \tc{G} ) } \circ \phi = \tau_{ C^{*} ( \tc{E} ) }.
\]
The homeomorphism induces a lattice isomorphism $\Phi \colon \Her(\tc{G}) \to \Her(\tc{E})$.

We want to define a bijection $\ftn{\psi}{\tc{G}^0}{\tc{E}^0}$ such that $\Phi(\iota_{\tc{G}}(u)) = \iota_{\tc{E}}(\phi(u))$.
Fix a vertex $u \in \tc{G}^0$. 
By Corollary \ref{iotalatticeiso} there is at least one vertex $\tilde{u} \in \tc{E}^0$ such that $\Phi(\iota_{\tc{G}}(u)) = \iota_{\tc{E}}(\tilde{u})$. 
Using the fact that $\tau_{ C^{*} ( \tc{G} ) } \circ \phi = \tau_{ C^{*} ( \tc{E} ) }$, we will prove that the sets $\langle u \rangle_{\tc{G}}$ and $\langle \tilde{u} \rangle_{\tc{E}}$ are the same size.
Indeed, considering first $H=\iota_{\tc{G}}(u)$, we note that by Lemma \ref{desciotaimg} it contains a largest proper hereditary subset $H_0$. 
Let $x_{u} \in \mathrm{Prim} ( C^{*} ( \tc{G} ) )$ be the unique element such that $C^{*} ( \tc{G} ) [ x_{u} ] = I_{H}^{ \tc{G} } / I_{ H_{0}}^{\tc{G} }$.
Note that by Lemma \ref{simplesubq}
\[
	| \tau_{ C^{*} ( \tc{G} ) } ( x_{u} ) | =  | \mathrm{rank} ( K_{0} ( C^{*} ( \tc{G} ) [ x_{u}  ] ) ) | = | \mathrm{rank} ( K_{0} ( I_{H}^{ \tc{G} } / I_{ H_{0}}^{\tc{G} } ) ) | = | \langle u \rangle_{ \tc{G} } |.
\]

Let $F = \iota_{\tc{E}} ( \tilde{u} )$ and let $F_{0}$ be the largest proper hereditary subset of $F$.
Note that since $\Phi$ is a lattice isomorphism and since $H_{0}$ is the largest proper hereditary subset of $H = \iota_{ \tc{G} } (u)$, $\Phi ( H_{0} )$ is the largest proper hereditary subset of $\Phi (H) = \iota_{ \tc{E} } ( \widetilde{u} )$.
That is $\Phi ( H_{0} ) = F_{0}$.
Therefore, $I_{F}^{ \tc{E} } / I_{ F_{0}}^{\tc{E} } = C^{*} ( \tc{E} ) [  \phi ( x_{u} )  ]$.
Computing as before, we then get
\[
	| \langle \tilde{u} \rangle_{ \tc{E} } | = |\tau_{ C^{*} ( \tc{E} ) } ( \phi (x_{u}) )| = |\tau_{ C^{*} ( \tc{G} ) } ( x_{u} )| = | \langle u \rangle_{ \tc{G} } |.
\] 
Thus we can define a bijection $\psi \colon \tc{G}^0 \to \tc{E}^0$ that satisfies
\[
	\Phi(\iota_{\tc{G}}(u)) = \iota_{\tc{E}}(\psi(u)).
\]	

We claim that $\psi$ can be used to construct a graph isomorphism. 
Since we are dealing with amplifications of transitively closed graphs, all we need to show is that if $u,v \in \tc{G}^0$ then $u \geq v$ if and only if $\psi(u) \geq \psi(v)$, and that there is a simple loop based at $u$ if and only if there is a simple loop based at $\psi(u)$.

Suppose $u \neq v$.
Then
\begin{align*}
	u \geq v	& \iff \iota_{\tc{G}}(u) \supseteq \iota_{\tc{G}}(v) \\
			& \iff \Phi(\iota_{\tc{G}}(u)) \supseteq \Phi(\iota_{\tc{G}}(v)) \\
			& \iff \iota_{\tc{E}}(\psi(u)) \geq \iota_{\tc{E}}(\psi(v)) \\
			& \iff \psi(u) \geq \psi(v).
\end{align*}
So we are left with checking the claim about simple loops.
If $|\langle u \rangle| \geq 2$ then there is some vertex $w \in \tc{G}^0$ such that $w \neq u$ and $u \geq w \geq u$. 
So since $\tc{G}$ is transitively closed, there is a simple loop based at $u$. 
Since $|\langle u \rangle| = |\langle \psi(u) \rangle|$, the same argument shows that there is a simple loop based at $\psi(u)$.

We now consider the case $|\langle u \rangle| = 1$. 
In this case $C^*(\tc{G}_u)$ is stably isomorphic to either $\mathcal{O}_\infty$ or $\C$ depending on whether or not there is a simple loop based at $u$. 
Similarly for $C^*(\tc{E}_{\psi(u)})$. 
Let $x_{u} \in \mathrm{Prim} ( C^{*} ( \tc{G} ) )$ be the unique element such that $C^{*} ( \tc{G} ) [ x_{u} ] = I_{H}^{ \tc{G} } / I_{ H_{0}}^{\tc{G} }$.
As before we see that 
\[
	\tau_{ C^{*} ( \tc{E} ) } ( \phi (x_{u}) ) = \tau_{ C^{*} ( \tc{G} ) } ( x_{u} ).
\]
Hence the simple sub-quotients are either both stably isomorphic to $\mathcal{O}_\infty$ or both stably isomorphic to $\C$ (depending on the sign of $\tau(\cdots)$).
Therefore there is a simple loop based at $u$ if and only if there is one based at $\psi(u)$ if and only if $C^*(\tc{G}_u)$ is stably isomorphic to $\mathcal{O}_\infty$. 

We can now extend $\psi$ to an isomorphism between $\tc{G}$ and $\tc{E}$.
\end{proof}

\begin{theor}\label{ourmain}
Let $G$ and $F$ be graphs with finitely many vertices.
The following are equivalent.
\begin{itemize}
\item[(i)] $\tc{G} \cong \tc{F}$.

\item[(ii)] $C^{*} ( \tc{G} ) \cong C^{*} ( \tc{F} )$.

\item[(iii)] $C^{*} ( \overline{G} ) \cong C^{*} (\overline{F} )$.

\item[(iv)] $C^{*} ( \overline{G} ) \otimes \K \cong C^{*} ( \overline{F} ) \otimes \K $.

\item[(v)] $\FK( C^{*} ( \overline{G} ) ) \cong \FK( C^{*} ( \overline{F} )  )$.

\item[(vi)] $\primT( C^{*} ( \overline{G} ) ) \cong \primT( C^{*} ( \overline{F} )  )$.
\end{itemize}
\end{theor}
\begin{proof}
Clearly $(i) \implies (ii)$. 
By Corollary \ref{transitiveClosure}, $(ii) \implies (iii)$.
The implications $(iii) \implies (iv)$, $(iv) \implies (v)$, and $(v) \implies (vi)$ all hold for obvious reasons. 
Finally the implication $(vi) \implies (i)$ follows by first appealing to Corollary \ref{transitiveClosure} to see that 
\[
	\primT( C^{*} ( \tc{G} ) ) \cong \primT( C^{*} ( \tc{F} )  ),
\]
and then applying Proposition \ref{p:invtrancls}.
\end{proof}

We can even do strong classification. 

\begin{lemma} \label{l:strongtc}
Let $G$ and $F$ be graphs with finitely many vertices.
If we are given an isomorphism $\alpha \colon \primT(C^*(\tc{G})) \to \primT(C^*(\tc{F}))$, then we can find an isomrphism $\phi \colon C^*(\tc{G}) \to C^*(\tc{F})$ such that $\primT(\phi) = \alpha$.
\end{lemma}
\begin{proof}
Let $\beta \colon \Her(\tc{G}) \to \Her(\tc{F})$ denote the lattice isomorphism induced by $\alpha$.
Recall from the proof of Proposition \ref{p:invtrancls}, that we can find a graph isomorphism  $\psi \colon \tc{G} \to \tc{F}$, such that if $u,v \in \tc{G}^0$ then $u \geq v$ if and only if $\psi(u) \geq \psi(v)$. 
Hence the lattice isomorphism from $\Her(\tc{G})$ to $\Her(\tc{F})$ induced by $\psi$ is the same as $\beta$.
Therefore the $C^*$-isomorphism induced by $\psi$ will induce $\alpha$.
\end{proof}

\begin{propo} \label{p:strongclassamp}
Let $G$ and $F$ be graphs with finitely many vertices.
If we are given an isomorphism $\alpha \colon \primT(C^*(\overline{G})) \to \primT(C^*(\overline{F}))$, then we can find an isomorphism $\phi \colon C^*(\overline{G}) \to C^*(\overline{F})$ such that $\primT(\phi) = \alpha$.
\end{propo}
\begin{proof}
By Corollary \ref{transitiveClosure} we can find $*$-isomorphisms
\[
	 \phi \colon C^*(\overline{G}) \to C^*(\tc{G}) \quad \text{ and } \quad \psi \colon C^*(\overline{F}) \to C^*(\tc{F}).
\]
Let 
\[
	\beta = \primT(\psi) \circ \alpha \circ \primT(\phi^{-1}).
\]
Then $\beta$ is an isomorphism from $\primT(C^*(\tc{G}))$ to $\primT(C^*(\tc{F}))$. 
So by Lemma \ref{l:strongtc} we can find a $*$-isomorphism $\chi \colon C^*(\tc{G}) \to C^*(\tc{F})$ such that $\primT(\chi) = \beta$.
We now put
\[
	\lambda = \psi^{-1} \circ \chi \circ \phi.
\]
Note that $\lambda$ is an isomorphism from $C^*(\overline{G})$ to $C^*(\overline{F})$ and that 
\begin{align*}
	\primT(\lambda)	& = \primT(\psi^{-1}) \circ \primT(\chi) \circ \primT(\phi) 
				   = \primT(\psi^{-1}) \circ \beta \circ \primT(\phi) \\
				& = \primT(\psi^{-1}) \circ \primT(\psi) \circ \alpha \circ \primT(\phi^{-1}) \circ \primT(\phi)
				   = \alpha. \qedhere
\end{align*}
\end{proof}

\begin{examp}
The graphs given by matrices
\[
\begin{bmatrix}
0&\infty&0&0&0\\
\infty&0&\infty&0&0\\
0&0&0&\infty&0\\
0&0&0&0&\infty\\
0&0&0&0&\infty\\
\end{bmatrix}
\qquad
\begin{bmatrix}
\infty&\infty&\infty&\infty&\infty\\
\infty&\infty&\infty&\infty&\infty\\
0&0&0&\infty&\infty\\
0&0&0&0&\infty\\
0&0&0&0&\infty\\
\end{bmatrix}
\]
were considered and proved to give stably isomorphic $C^*$-algebras in \cite{ERRlinear}. We now know that their algebras are in fact isomorphic. Note also that the second graph is the transitive closure of the first; indeed examples of this type inspired the results in the present paper.
\end{examp}

\section{Classification of $C^{*}$-algebras over $X$ with free $K$-theory}

In this section, we show that $\ouri_{ X , \Sigma } ( - )$ is in fact a complete invariant for an a priori much larger class of $C^{*}$-algebras, containing the class of $C^{*}$-algebras associated to amplified graphs.  
We will use the generalized classification results in this section to prove permanence properties in the the next section.   

\begin{defin}\label{d:freeclass}
Let $\mathcal{C}$ be the class of separable, nuclear, simple, purely infinite $C^{*}$-algebras $\mathfrak{A}$ satisfying the UCT such that $K_{1} ( \mathfrak{A} ) = 0$ and $K_{0} ( \mathfrak{A} )$ is free.  

Let $\mathcal{C}_{ \mathrm{free} }$ be the class of $C^{*}$-algebras $\mathfrak{A}$ such that 
\begin{itemize}
\item[(1)] $\mathrm{Prim} ( \mathfrak{A} )$ is finite; and

\item[(2)] For each $x \in \mathrm{Prim} ( \mathfrak{A} )$, $\mathfrak{A} [ x  ]$ is unital or stable with $\mathfrak{A} [ x  ]$ in $\mathcal{C}$ or stably isomorphic to $\K$.

\item[(3)] For each $x \in \mathrm{Prim} ( \mathfrak{A} )$, if $\mathfrak{A} [ x  ]$ is unital, then there exists an isomorphism $K_{0} ( \mathfrak{A} [ x ] ) \cong \bigoplus_{ n } \Z$ such that $[ 1_{ \mathfrak{A} [ x ] } ]$ is sent to $( 1, \lambda )$. 
\end{itemize}
\end{defin}

\begin{remar}\label{r:elements}
Let $\mathfrak{A}$ be a $C^{*}$-algebra in $\mathcal{C}_{ \mathrm{free} }$.  Let $x \in \mathrm{Prim} ( \mathfrak{A} )$.  Suppose $\mathfrak{A} [ x ]$ is unital and not in $\mathcal{C}$.  Since $\mathfrak{A} [ x ]$ is stably isomorphic to $\K$, $\mathfrak{A} [ x  ] \cong \mathsf{M}_{n}$ for some $n \in \N$.  Condition (3) implies that $\mathfrak{A} [ x ] = \C$.
\end{remar}

\begin{propo}\label{p:elements}
Let $\mathfrak{A}$ be a $C^{*}$-algebra with $\mathrm{Prim} ( \mathfrak{A} )$ finite.  
\begin{itemize}
\item[(1)] $\mathfrak{A}$ in $\mathcal{C}_{ \mathrm{free} }$ if and only if for each $x \in  \mathrm{Prim} ( \mathfrak{A} )$, $\mathfrak{A}[x]$ is in $\mathcal{C}_{ \mathrm{free} }$.

\item[(2)] If $\mathfrak{A}$ in $\mathcal{C}_{ \mathrm{free} }$, then for each $Y \in \mathbb{LC} ( \mathrm{Prim} ( \mathfrak{A} ) )$, $\mathfrak{A}[Y]$ is in $\mathcal{C}_{ \mathrm{free} }$.

\item[(3)]  If $\mathfrak{A}$ in $\mathcal{C}_{ \mathrm{free} }$, then $\mathfrak{A} \otimes \K$ is in $\mathcal{C}_{ \mathrm{free} }$.
\end{itemize}
\end{propo}

\begin{proof}
(1) is clear from Definition \ref{d:freeclass}.  We now prove (2).  Suppose $\mathfrak{A}$ in $\mathcal{C}_{ \mathrm{free} }$.  Set $X = \mathrm{Prim} ( \mathfrak{A} )$.  Let $Y \in \mathbb{LC} ( X ) $.  Then $Y = U \setminus V$ with $U, V \in \mathbb{O} ( X )$ such that $V \subseteq U$.  Note that $\mathfrak{B} = \mathfrak{A} [ Y ]$ is a tight $C^{*}$-algebra over $Y$.  Since $Y$ is finite, $\mathrm{Prim} ( \mathfrak{B} )$ is finite.  


Since $\mathbb{LC} ( Y ) \subseteq \mathbb{LC} ( X )$, we get that for each $s \in \mathrm{Prim} ( \mathfrak{B} )$, $\mathfrak{B} [ s ] \cong \mathfrak{A} [ x_{s} ]$ where $x_{s} \in Y$.  Since $\mathfrak{A} [ x_{s} ] \in \mathcal{C}_{ \mathrm{free} }$, $\mathfrak{B} [s] \in \mathcal{C}_{ \mathrm{free} }$.  Thus, by (1), $\mathfrak{B} \in \mathcal{C}_{ \mathrm{free} }$.

For (3), first note that $\mathrm{Prim} ( \mathfrak{A} )$ is homeomorphic to $\mathrm{Prim} ( \mathfrak{A} \otimes \K )$.  So, $\mathrm{Prim} ( \mathfrak{A} \otimes \K )$ is finite.  By Corollary 2.3 of \cite{RorStableRev}, every quotient and ideal of a stable $C^{*}$-algebra is stable.  Thus, for every $x \in X$, we have that $( \mathfrak{A} \otimes \K )[ x ]$ is stable.    Thus, (3) in Definition \ref{d:freeclass} is vacuously true. Note that for each $x \in X$, $\mathfrak{A}[x] \otimes \K \cong ( \mathfrak{A} \otimes \K ) [ x ]$.  Thus, Condition (2) in Definition \ref{d:freeclass} holds.  We have just shown that $\mathfrak{A} \otimes \K \in \mathcal{C}_{\mathrm{free} }$.
\end{proof}

\subsection{Singular graph $C^{*}$-algebras are in $\mathcal{C}_{\mathrm{free}}$}

We will now show that $\mathcal{C}_{\mathrm{free}}$ actually contains the algebras we are interested in. 
In fact we will show that it even contains the graph algebras of a larger class of graphs than the amplified, namely the singular graphs with no breaking vertices. 
This will be done in a sequence of small steps.

\begin{lemma}\label{l:sinkinfemit}
Let $G$ be a singular graph.  
If $C^{*} ( G )$ is simple then either $G^{0} = \{ v_{0} \}$ and $G^{1} = \emptyset$ or $G$ contains a cycle.
In the case that $G$ contains a cycle, we have that $G$ is strongly connected.
\end{lemma}

\begin{proof}
Suppose $G$ contains no cycle. 
Let $H = \{ v \in G^0 \mid v_0 \geq v \} \setminus \{ v_0 \}$.
Since there are no cycles in $G$, $H$ is hereditary. 
As $C^*(G)$ is simple and $H \neq G^0$ we must have that $H$ is empty. 
Thus $v_0$ is a sink. 
In particular the set $\{ v_0 \}$ is non-empty and hereditary.
Using again that $C^*(G)$ is simple we deduce that $G^0 = \{ v_0 \}$.

Suppose $G$ contains a cycle.
Since every vertex in $G$ is singular, by Corollary 2.15 of \cite{ddmt_arbgraph}, if $v, w \in G^{0}$, then $v \geq w$ and $w \geq v$.
Hence, $G$ is strongly connected.
\end{proof}

\begin{lemma}\label{l:simplecase}
Let $G$ be a singular graph such that $C^{*} ( G )$ is simple.
Then $C^{*} ( G ) \in \mathcal{C}_{\mathrm{free} }$.
\end{lemma}
\begin{proof}
By Lemma \ref{l:sinkinfemit} and \cite{ddmt_arbgraph}, $C^{*} ( G ) \cong \C$ or $C^{*} ( G )$ is purely infinite.
By Corollary 3.2 of \cite{ddmt_kthygraph} and Theorem 2.2 of \cite{mt:orderkthy}, $K_{0} ( C^{*} ( G ) ) \cong \bigoplus_{ G^{0} } \Z$ via an isomorphism that maps $[ 1_{ C^{*}(G )} ]$ to $\bigoplus_{ G^{0} } 1$, and $K_{1} ( C^{*} ( G ) ) = 0$. 
\end{proof}

The following definition is useful when dealing with breaking vertices.

\begin{defin}
Let $G$ be a graph and let $H$ be a hereditary subset of $G^{0}$.
Set
\begin{equation*}
F_{H} = \setof{ \alpha \in G^{*} }{ \text{$s_G ( \alpha ) \notin H$, $r_G ( \alpha ) \in H$, and $r_{G} ( \alpha_{i} ) \notin H$ for $i < | \alpha |$ } }
\end{equation*} 
Set
\begin{equation*}
_{H}G_{\emptyset}^{0} = H \cup F_{H} \quad \text{and} \quad _{H}G_{\emptyset }^{1} = s_{G}^{-1} ( H )  \cup \setof{ \overline{\alpha} }{ \alpha \in F_{H} }.
\end{equation*}
Define $(s_{ _{H}G_{\emptyset}  }) \vert_{ s_{G}^{-1} ( H ) } = (s_{G}) \vert_{s_{G}^{-1} ( H ) }$, $s_{ _{H}G_{\emptyset}  } ( \overline{\alpha} ) = \alpha$, $(r_{ _{H}G_{\emptyset}  } )\vert_{ s_{G}^{-1} ( H ) } = (r_{G} )\vert_{s_{G}^{-1} ( H ) }$, $r_{ _{H}G_{\emptyset}  } ( \overline{\alpha} ) = r( \alpha )$.  
\end{defin} 

\begin{lemma} \label{l:sizeoffh}
Let $G$ be a singular graph with finitely many vertices and no breaking vertices. 
If $H \subseteq G^0$ is hereditary then $F_H$ is either infinite or empty. 
\end{lemma}
\begin{proof}
Suppose $F_H$ is non-empty. 
Then there is some vertex $v \notin H$ such that $r(s^{-1}(v)) \cap H$ is non-empty. 
If $r(s^{-1}(v)) \cap H$ is infinite then clearly $F_H$ is infinite. 
So suppose $r(s^{-1}(v)) \cap H$ is a finite set. 
Since $v$ is an infinite emitter and $G^{0}$ is finite there is some vertex $u \notin H$ that $v$ emits infinitely to. 
Let $K = \{ w \in G^0 \mid u \geq w \}$. 
Clearly $K$ is hereditary.  
As there are no breaking vertices in $G$ we must have that $r(s^{-1}(v)) \cap H \subseteq K$. 
Since neither $v$ nor $u$ are in $H$ there are infinitely many paths $\alpha$ such that $s(\alpha) = u$, $r(\alpha) \in H$, and $r( \alpha_{i} ) \notin H$ for $i < | \alpha |$.
Hence $F_H$ is infinite. 
\end{proof}

\begin{lemma} \label{l:idealunitalorstable}
Let $G$ be a singular graph with finitely many vertices and no breaking vertices. 
Let $H$ be a hereditary subset of $G^0$ such that $I_H$ is simple. 
Put $E = (H, s^{-1}(H), s \vert_H, r \vert_H)$. 
Either $I_H \cong C^*(E)$ or $I_H \cong C^*(E) \otimes \K$. 
\end{lemma}
\begin{proof}
By Lemma \ref{l:sizeoffh} $F_H$ is either infinite or empty. 
If it is empty then $_{H}G_{\emptyset} \cong E$ and so $I_H \cong C^*(_{H}G_{\emptyset}) \cong C^*(E)$. 

If $F_H$ is infinite then since $H$ is finite there must be some vertex in $_{H}G_{\emptyset}$ that receives edges from infinitely many other vertices. 
So by  Lemma 6.3 of \cite{semt_classgraphalg} $I_H$ is stable.
Hence 
\[
	I_H \cong I_H \otimes \K \cong C^*(E) \otimes \K. \qedhere
\]	
\end{proof}

\begin{lemma} \label{l:simpleideal}
Let $G$ be a singular graph with finitely many vertices and no breaking vertices. 
Let $H \subseteq G^0$ be non-empty and hereditary. 
If $I_H$ is simple then $I_H \in  \mathcal{C}_{\mathrm{free} }$
\end{lemma}
\begin{proof}
Put $E = (H, s^{-1}(H), s \vert_H, r \vert_H)$. 
Since $G$ is singular so is $E$.
Hence Lemma $\ref{l:simplecase}$ implies that $C^*(E)$ is in $\mathcal{C}_{\mathrm{free} }$.
By Proposition \ref{p:elements} so is $C^*(E) \otimes \K$. 
Lemma \ref{l:idealunitalorstable} says that $I_H$ is isomorphic to either $C^*(E)$ or $C^*(E) \otimes \K$, so it is in $\mathcal{C}_{\mathrm{free} }$.
\end{proof}

\begin{propo}\label{p:apgraphclass}
If $G$ is a singular graph with finitely many vertices and no breaking vertices then $C^{*} ( G ) \in \mathcal{C}_{ \mathrm{free} }$.
\end{propo}
\begin{proof}
Since $G^0$ is finite we have that $\mathrm{Prim} (C^*(G))$ is finite. 
By Proposition \ref{p:elements} it suffices to show that every simple sub-quotient of $C^*(G)$ is in $\mathcal{C}_{ \mathrm{free} }$. 
A simple sub-quotient is a simple ideal in a quotient. 
We have shown in Lemma \ref{l:simpleideal} that all simple ideals of singular graph algebras are in  $\mathcal{C}_{ \mathrm{free} }$. 
Hence if we can show that any quotient of $C^*(G)$ is a singular graph algebra we will be done. 

Suppose that $H$ is a hereditary subset of $G^0$. 
Define a graph $G / H$ by $G / H = (G^0 \setminus H, r^{-1}(G^0 \setminus H), r, s)$. 
We have that $C^*(G) / I_H \cong C^*(G / H)$. 
Since there are no breaking vertices in $G$ any vertex that maps infinitely into $H$ will also map infinitely to $G^0 \setminus H$, so $G / H$ is singular and has no breaking vertices. 
\end{proof}

\subsection{Classification of $C^{*}$-algebras in $\mathcal{C}_{\mathrm{free}}$}

We are now ready to prove that $\ouri_{ X , \Sigma } ( - )$ is a complete invariant for $C^{*}$-algebras in $\mathcal{C}_{\mathrm{free}}$, see Theorem \ref{t:classification}.  Although the proof of Theorem \ref{t:classification} is quite technical, the techniques are similar to that of the proof of Theorem 3.9 of \cite{ERRshift}.

The following definition is taken from \cite[Definition 3.3]{HL_unitaryequiv}.  

\begin{defin}
Let $\mathfrak{A}$ and $\mathfrak{B}$ be $C^{*}$-algebras such that $\mathfrak{A}$ is unital. 
Let $H_{1} ( K_{0} ( \mathfrak{A} ), K_{0} ( \mathfrak{B} ) )$ be the subgroup of $K_{0} ( \mathfrak{B} )$ consisting of all elements $x \in K_{0} ( \mathfrak{B} )$ such that there exists a group homomorphism $\ftn{ \alpha }{ K_{0} ( \mathfrak{A} ) }{ K_{0} ( \mathfrak{B} ) }$ with $\alpha ( [ 1_{ \mathfrak{A} } ] ) = x$.
\end{defin}

\begin{lemma}\label{l:grplin}
Let $\mathfrak{A}$ be a unital $C^{*}$-algebra.  
If there exists a homomorphism $\ftn{ \beta }{ K_{0} ( \mathfrak{A} ) }{ \Z }$ such that $\beta ( [ 1_{ \mathfrak{A} } ] ) = 1$, then for any $C^{*}$-algebra $\mathfrak{B}$, $H_{1} ( K_{0} ( \mathfrak{A} ), K_{0} ( \mathfrak{B} ) ) = K_{0} ( \mathfrak{B} )$.  
Consequently, if $K_{0} ( \mathfrak{A} ) \cong \Z \oplus G$ where the isomorphism sends $[ 1_{ \mathfrak{A} } ]$ to $( 1 , x )$, then for any $C^{*}$-algebra $\mathfrak{B}$, $H_{1} ( K_{0} ( \mathfrak{A} ), K_{0} ( \mathfrak{B} ) ) = K_{0} ( \mathfrak{B} )$.  
\end{lemma}

\begin{proof}
Let $x \in K_{0} ( \mathfrak{B} )$.
Define $\ftn{ \gamma }{ \Z }{ K_{0} ( \mathfrak{B} ) }$ by $\gamma ( n ) = n x$.
Then $\beta \circ \gamma$ is homomorphism from $K_{0} ( \mathfrak{A} )$ to $K_{0} ( \mathfrak{B} )$ with $( \beta \circ \gamma ) ( [ 1_{ \mathfrak{A} } ] ) = x$.  
\end{proof}

The following lemma is one of the key technical lemma to prove our classification result.  It provides a way to get an isomorphism between two extensions given that the ideals are isomorphic and quotients are isomorphic.

\begin{lemma}\label{lem:class}
For $i = 1,2$, let $\mathfrak{e}_{i} : \ 0 \to \mathfrak{I}_{i} \to \mathfrak{E}_{i} \to \mathfrak{A}_{i} \to 0$ be a full extension.
Suppose $\mathfrak{I}_{i}$ is a stable $C^{*}$-algebra satisfying the corona factorization property.
Suppose there exist an isomorphism $\ftn{ \phi_{0} }{ \mathfrak{I}_{1} }{ \mathfrak{I}_{2} }$ and an isomorphism $\ftn{ \phi_{2} }{ \mathfrak{A}_{1} }{ \mathfrak{A}_{2} }$ such that $\kk ( \phi_{2} ) \times [ \tau_{ \mathfrak{e}_{2} } ] = [ \tau_{ \mathfrak{e}_{1} } ] \times \kk ( \phi_{0} )$.

\begin{itemize}
\item[(a)] If $\mathfrak{e}_{1}$ and $\mathfrak{e}_{2}$ are non-unital extensions, then there exists an isomorphism $\ftn{ \phi_{1} }{ \mathfrak{A}_{1} }{ \mathfrak{A}_{2} }$ such that $\pi_{2} \circ \phi_{1} = \phi_{2} \circ \pi_{1}$.  

\item[(b)] If $\mathfrak{e}_{1}$ and $\mathfrak{e}_{2}$ are unital extensions and $K_{0} ( \mathfrak{A}_{1} ) \cong \Z \oplus G$ with $[1_{ \mathfrak{A}_{1} } ]$ mapping to $(1, x )$, then there exists an isomorphism $\ftn{ \phi_{1} }{ \mathfrak{A}_{1} }{ \mathfrak{A}_{2} }$ such that $\pi_{2} \circ \phi_{1} = \phi_{2} \circ \pi_{1}$.
\end{itemize}
\end{lemma}

\begin{proof}
Since $[ \tau_{ \mathfrak{e}_{1} \cdot \phi_{0} } ] = [ \tau_{ \mathfrak{e}_{1} } ] \times \kk ( \phi_{0} ) = \kk ( \phi_{2} ) \times [ \tau_{ \mathfrak{e}_{2} } ] = [ \tau_{ \phi_{2} \cdot \mathfrak{e}_{2} } ]$ in $\kk^{1} ( \mathfrak{A}_{1} , \mathfrak{I}_{2} )$, we have that $[ \tau_{ \mathfrak{e}_{1} \cdot \phi_{0} } ] = [ \tau_{ \phi_{2} \cdot \mathfrak{e}_{2} } ]$.  

Suppose $\mathfrak{e}_{1}$ and $\mathfrak{e}_{2}$ are non-unital extensions.  Then $\tau_{ \mathfrak{e}_{1} \cdot \phi_{0} }$ and $\tau_{ \phi_{2} \cdot \mathfrak{e}_{2} }$ are non-unital full extensions.  Since $\mathfrak{I}_{2}$ satisfies the corona factorization property, by Theorem 3.2(2) of \cite{NgCFP}, there exists a unitary $u$ in $\mathcal{M}( \mathfrak{I}_{2} )$ such that $\mathrm{Ad} ( \pi ( u ) ) \circ \tau_{ \mathfrak{e}_{1} \cdot \phi_{0} } = \tau_{ \phi_{2} \cdot \mathfrak{e}_{2} }$.  

Suppose $\mathfrak{e}_{1}$ and $\mathfrak{e}_{2}$ are unital extensions with $K_{0} ( \mathfrak{A}_{1} ) \cong \Z \oplus G$ with $[1_{ \mathfrak{A}_{1} } ]$ mapping to $(1, x )$.  Then $\tau_{ \mathfrak{e}_{1} \cdot \phi_{0} }$ and $\tau_{ \phi_{2} \cdot \mathfrak{e}_{2} }$ are unital full extensions.  By Theorem 2.4 and Corollary 3.8 in \cite{HL_unitaryequiv} and Lemma \ref{l:grplin}, there exists a unitary $u$ in $\mathcal{M}( \mathfrak{I}_{2} )$ such that $\mathrm{Ad} ( \pi ( u ) ) \circ \tau_{ \mathfrak{e}_{1} \cdot \phi_{0} } = \tau_{ \phi_{2} \cdot \mathfrak{e}_{2} }$.  

In both cases, there exists $u \in \mathfrak{I}_{2}$ such that $\mathrm{Ad} ( \pi ( u ) ) \circ \tau_{ \mathfrak{e}_{1} \cdot \phi_{0} } = \tau_{ \phi_{2} \cdot \mathfrak{e}_{2} }$.  Hence, the triple $( \mathrm{Ad} ( u ), \mathrm{Ad} (u) , \id_{ \mathfrak{A}_{1} } )$ is an isomorphism between $\mathfrak{e}_{1} \cdot \phi_{0}$ and $\phi_{2} \cdot \mathfrak{e}_{2}$.   Therefore, $\mathfrak{e}_{1}$ is isomorphic to $\mathfrak{e}_{2}$.
\end{proof}

\begin{lemma}\label{l:equivariant}
Let $\mathfrak{A}$ and $\mathfrak{B}$ be tight $C^{*}$-algebras over $X$.
Let $U \in \mathbb{O} ( X )$.
Suppose there exist isomorphisms $\ftn{ \phi_{ 0 } }{ \mathfrak{A} [  U ] }{ \mathfrak{B}[U] }$, $\ftn{ \phi_{1} }{ \mathfrak{A} }{ \mathfrak{B} }$, and $\ftn{ \phi_{ 0 } }{ \mathfrak{A} / \mathfrak{A} [  U ] }{ \mathfrak{B} / \mathfrak{B}[U] }$ such that 
\begin{equation*}
\xymatrix{
0 \ar[r] & \mathfrak{A} [U] \ar[r]^-{ \iota_{1} } \ar[d]^{ \phi_{0} } & \mathfrak{A} \ar[r]^-{ \pi_{1} } \ar[d]^{ \phi_{1} } & \mathfrak{A} / \mathfrak{A} [U ] \ar[r] \ar[d]^{\phi_{2} } & 0 \\
0 \ar[r] & \mathfrak{B} [U] \ar[r]_-{ \iota_{2} }  & \mathfrak{B} \ar[r]_-{ \pi_{2} }  & \mathfrak{B} / \mathfrak{B} [  U ] \ar[r]  & 0
}
\end{equation*}
commutes.
$\varphi_{1}$ is an $X$-equivariant isomorphism if and only if $\varphi_{0}$ is a $U$-equivariant isomorphism and $\varphi_{2}$ is an $X \setminus U$-equivariant isomorphism.
\end{lemma}

\begin{proof}
Suppose $\varphi_{1}$ is an $X$-equivariant.  Let $V$ be an open subset of $U$.  Then $V \in \mathbb{O}(X)$ since $U \in \mathbb{O}(X)$.  Hence, $\varphi_{0} ( \mathfrak{A}[V] ) = \varphi_{1} ( \mathfrak{A} [V] ) = \mathfrak{B} [V ]$.  Hence, $\varphi_{0}$ is a $U$-equivariant isomorphism.  

Suppose $Y \in \mathbb{O} ( X \setminus U )$.  Then $Y = V \setminus U$ where $V \in \mathbb{O} ( X )$ and $U \subseteq V$.  Thus, 
\[
	\varphi_{2} ( (\mathfrak{A} / \mathfrak{A}[U]) [ Y] ) = \varphi_{2} ( \mathfrak{A}[V] / \mathfrak{A} [U] ) = \pi_{2} ( \mathfrak{B}[V] ) = \mathfrak{B} [V] / \mathfrak{B} [U] = (\mathfrak{B} / \mathfrak{B} [U]) [ Y ].
\]
Hence, $\phi_{2}$ is an $X\setminus U$-equivariant isomorphism.

Suppose $\varphi_{0}$ is a $U$-equivariant isomorphism and $\varphi_{2}$ is an $X \setminus U$-equivariant isomorphism.  Let $V \in \mathbb{O} (X)$ such that $V$ is not a subset of $U$.  Then $\varphi_{2} ( \mathfrak{A} [V \cup U ] / \mathfrak{A} [U] ) = \mathfrak{B} [ V \cup U ] / \mathfrak{B} [U]$.  Let $V_{1} \in \mathbb{O} ( X )$ such that $\varphi_{1} ( \mathfrak{A} [U] ) = \mathfrak{B} [V_{1} ]$.  Then
\[
	\mathfrak{B} [ (V_{1} \cup U) \setminus U ] = \pi_{2} ( \mathfrak{B}[V_{1}] ) = \mathfrak{B} [ ( V \cup U ) \setminus U ].
\]
   Hence, 
\[
V_{1} \setminus U  = ( V_{1} \cup U ) \setminus U 
				=  ( V \cup U ) \setminus U 
				= V \setminus U   
\]
Also, note that 
\[
\mathfrak{B} [ V_{1} \cap U ] = \mathfrak{B} [V_{1} ] \cap \mathfrak{B} [ U ] 
					= \phi_{1} ( \mathfrak{A} [ V ] \cap \mathfrak{A}[ U ] ) 
					= \phi_{1} ( \mathfrak{A} [ V \cap U ] ) 
					= \mathfrak{B} [ V \cap U ] 
\]
So, $V_{1} \cap U = V \cap U$.  Therefore, $V_{1} = V$.  

Suppose $V \in \mathbb{O} ( X )$ such that $V \subseteq U$.  Then $V \in \mathbb{O} ( U )$.  Hence, $\varphi_{1} ( \mathfrak{A} [V] ) = \varphi_{0} ( \mathfrak{A} [V] ) = \mathfrak{B} [V]$.  Hence, $\varphi_{1}$ is an $X$-equivariant isomorphism.
\end{proof}

\begin{lemma}\label{l:unit}
Let $m \in \N \cup \{ \infty \}$. 
Let $x \in \bigoplus_{ i  = 1 }^{m} \Z$ be given.
If $x_1 = 1$ then there exists an isomorphism $\ftn{ \alpha }{ \bigoplus_{ i = 1}^{m} \Z }{ \bigoplus_{ i=1 }^{m} \Z }$ such that $\alpha ( (1, 0 , \dots ) ) = x$.
\end{lemma}

\begin{proof}
Suppose $m \in \N$.
Define $\gamma = (\gamma_{ij}) \in \mathsf{M}_{m} ( \Z )$ as follows:
\begin{align*}
\gamma_{i,j} = 
\begin{cases}
	1,	& \text{if $i = j$, $1 \leq i \leq m$} \\
	x_i-1,	& \text{if $j = i-1$, $2 \leq i \leq m$} \\
	0,	& \text{otherwise}
\end{cases}
\end{align*}
Since $\gamma$ is a lower triangular matrix with $1$'s in the diagonal, we have that $\det( \gamma ) = 1$.
Hence, $\gamma$ is invertible in $\mathsf{M}_{m} ( \Z )$.
Therefore, $\ftn{ \gamma }{ \bigoplus_{ i = 1}^{m} \Z }{ \bigoplus_{ i =1 }^{m} \Z }$ is an isomorphism.
Note that $\gamma ( (1, 1, \dots, 1) )_{i} = x_{i}$ for each $i$.
Define $\delta = (\delta_{ij}) \in \mathsf{M}_{m} ( \Z )$ as follows:
\begin{align*}
\delta_{i,j} = 
\begin{cases}
	1,	& \text{if $i = j$, $1 \leq i \leq m$} \\
	-1,	& \text{if $j = i-1$, $2 \leq i \leq m$} \\
	0,	& \text{otherwise}
\end{cases}
\end{align*}
Since $\delta$ is a lower triangular matrix with $1$'s in the diagonal, we have that $\det( \delta ) = 1$.
Whence, $\delta$ is invertible in $\mathsf{M}_{m} ( \Z )$.
So $\ftn{ \delta }{ \bigoplus_{ i = 1}^{m} \Z }{ \bigoplus_{ i =1 }^{m} \Z }$ is an isomorphism with $\delta ( (1, 1, \ldots, 1 ) ) = ( 1, 0, \ldots,0)$.

Set $\alpha =  \gamma \circ \delta^{-1}$.
Then $\ftn{ \alpha }{ \bigoplus_{ i = 1}^{m} \Z }{ \bigoplus_{ i=1 }^{m} \Z }$ is an isomorphism such that 
\[
	\alpha((1,0,\ldots,0)) = \gamma(\delta^{-1}((1,0,\ldots,0))) = \gamma((1,1,\ldots, 1)) = x.
\]

Suppose $m = \infty$.
Then $x = ( x_{1}, x_{2}, \dots, x_{n} , 0 , 0 \dots )$.  
From the finite case, there exists an isomorphism $\ftn{ \alpha_{n} }{ \bigoplus_{i = 1 }^{n} \Z }{  \bigoplus_{i = 1 }^{n} \Z }$ such that $\alpha_{n} ( ( 1, 0, \dots, 0 ) ) = ( x_{1}, x_{2}, \dots, x_{n} )$.
Define $\ftn{ \beta }{  \bigoplus_{i = 1 }^{\infty} \Z }{ \bigoplus_{i = 1 }^{\infty} \Z }$ by 
\begin{align*}
\beta ( (y_{1}, y_{2}, \dots ) ) = ( \alpha_{n}^{-1} ( ( y_{1}, \dots, y_{n} ) ) , y_{n+1} , y_{n+2}, \dots )
\end{align*}
Since $\alpha_{n}$ is an isomorphism, we have that $\beta$ is an isomorphism.
Moreover,
\begin{align*}
\beta ( x ) = ( \alpha_{n}^{-1}( ( x_{1}, \dots, x_{n} ) ), 0, 0, \dots ) = ( 1,0, 0 \dots )
\end{align*}
Hence, $\ftn{ \alpha = \beta^{-1} }{  \bigoplus_{i = 1 }^{\infty} \Z }{ \bigoplus_{i = 1}^{\infty} \Z }$ is an isomorphism such that $\alpha ( (1, 0, \ldots ) ) = x$.
\end{proof}

\begin{lemma}\label{l:orderunit}
Suppose $\mathfrak{A}$ and $\mathfrak{B}$ are $C^{*}$-algebras in $\mathcal{C}_{ \mathrm{free} }$.  
Then $\primTS ( \mathfrak{A} ) \cong \primTS ( \mathfrak{B} )$ implies that $\mathfrak{A}[ x ]  \cong \mathfrak{B} [ \alpha (x) ] $ for all $x \in \mathrm{Prim} ( \mathfrak{A} )$, where $\ftn{ \alpha }{ \mathrm{Prim} ( \mathfrak{A} ) }{ \mathrm{Prim} ( \mathfrak{B} ) }$ is the homeomorphism implementing the isomorphism $\primTS ( \mathfrak{A} ) \cong \primTS ( \mathfrak{B} )$.
\end{lemma}

\begin{proof}
Let $x \in \mathrm{Prim} ( \mathfrak{A} )$. 
Since $\tau_{ \mathfrak{A} } = \tau_{ \mathfrak{B} } \circ \alpha$, we have that $K_{0} ( \mathfrak{A} [ x ]  ) \cong K_{0} ( \mathfrak{B}[x] )$ and that $\mathfrak{A} [ x ] $ and $\mathfrak{B} [ \alpha (x) ]$ are either both AF algebras or both purely infinite.
Since $\sigma_{ \mathfrak{A} } = \sigma_{ \mathfrak{B} } \circ \widetilde{\alpha}$, we have that $\mathfrak{A} [ x ] $ and $\mathfrak{B} [ \alpha (x) ] $ are either both stable or both unital.  

Suppose $\mathfrak{A} [ x ] $ and $\mathfrak{B} [ \alpha(x) ]$ are AF algebra.  
Then $\mathfrak{A}[x] \cong \K \cong \mathfrak{B}[ \alpha(x)]$ or $\mathfrak{A} [x]\cong \C \cong \mathfrak{B} [ \alpha(x)]$.  

Suppose $\mathfrak{A} [ x ] $ and $\mathfrak{B}[ \alpha(x) ]$ are purely infinite simple.
Since $K_{0} ( \mathfrak{A} [ x ]  ) \cong K_{0} ( \mathfrak{B}[ \alpha(x) ] )$, by \cite{kirchpure} we have that $\mathfrak{A} [ x ]  \otimes \K \cong \mathfrak{B} [ \alpha(x) ] \otimes \K$. 
Thus, if $\mathfrak{A} [ x ] $ and $\mathfrak{B} [ \alpha(x) ]$ are stable, $\mathfrak{A}[x] \cong \mathfrak{B}[ \alpha(x) ]$.  

We now consider the case where $\mathfrak{A} [ x ]$ and $\mathfrak{B} [ \alpha(x) ]$ are unital.
Then $K_{0} ( \mathfrak{A}[x] ) \cong \oplus_{k = 1}^{n} \Z$ such that the isomorphism sends $[ 1_{ \mathfrak{A}[x] } ]$ to $(1, y_{ \mathfrak{A} } )$ with $y_{ \mathfrak{A}} \in \oplus_{ k = 2}^{n} \Z$ and $K_{0} ( \mathfrak{B}[ \alpha(x)] ) \cong \oplus_{k = 1}^{n} \Z$ such that the isomorphism sends $[ 1_{ \mathfrak{B}[ \alpha(x) ] } ]$ to $(1, y_{ \mathfrak{B} } )$ with $y_{ \mathfrak{B} } \in \oplus_{ k = 2}^{n} \Z$.
From Lemma \ref{l:unit} we get automorphisms $\alpha$ and $\beta$  on $\oplus_{k = 1}^{n} \Z$ such that 
\[
	\alpha(1,0, \ldots, 0) = (1,y_{\mathfrak{A}}), \quad \text{ and } \quad \beta(1,0, \ldots, 0) = (1,y_{\mathfrak{B}}).
\]
Let $\gamma = \beta \circ \alpha^{-1}$. 
Then $\gamma$ is an automorphism of $\oplus_{k = 1}^{n} \Z$ which takes $(1,y_\mathfrak{A})$ to $(1,y_\mathfrak{B})$. Thus, by \cite{kirchpure}, $\mathfrak{A} [ x ] \cong \mathfrak{B} [ \alpha(x) ]$.
\end{proof}

\begin{theor}\label{t:classification}
Let $X$ be a finite topological space and let $\mathfrak{A}_1$ and $\mathfrak{A}_2$ be tight $C^{*}$-algebras over $X$. 
Suppose $\mathfrak{A}_1$ and $\mathfrak{A}_1$ are in $\mathcal{C}_{ \mathrm{free} }$.
There exists an $X$-equivariant isomorphism $\ftn{ \varphi }{ \mathfrak{A}_1 }{ \mathfrak{A}_2 }$ if and only if $\ouri_{X, \Sigma} ( \mathfrak{A}_1 ) \equiv \ouri_{X, \Sigma} ( \mathfrak{A}_2 )$. 
Consequently, there exists an $X$-equivariant isomorphism $\ftn{ \varphi }{ \mathfrak{A}_1 \otimes \K }{ \mathfrak{A}_2 \otimes \K }$ if and only if $\ouri_{X} ( \mathfrak{A}_1 ) \equiv \ouri_{X} ( \mathfrak{A}_2 )$.
\end{theor}

\begin{proof}
It is clear that if there exists an $X$-equivariant isomorphism $\ftn{ \varphi }{ \mathfrak{A}_1  }{ \mathfrak{A}_2  }$, then $\ouri_{X, \Sigma} ( \mathfrak{A}_1 ) \equiv \ouri_{X, \Sigma} ( \mathfrak{A}_2 )$.

For the converse, we will induct on $X$.  Suppose $X$ has one point.  Then $\mathfrak{A}_1$ and $\mathfrak{A}_2$ are simple $C^{*}$-algebras.  Suppose $\ouri_{X, \Sigma } ( \mathfrak{A}_1 ) \equiv \ouri_{X, \Sigma} ( \mathfrak{A}_2 )$.  By Lemma \ref{l:orderunit}, $\mathfrak{A}_1 \cong \mathfrak{A}_2$.  

Suppose the theorem is true for any finite topological space with less than or equal to $m-1$ elements, and that $X$ is a finite topological space with $m$ elements.  Suppose $\alpha : \ouri_{X, \Sigma} ( \mathfrak{A}_1 ) \equiv \ouri_{X, \Sigma } ( \mathfrak{A}_2 )$.  Note that if $X$ is disconnected then $\mathfrak{A}_1$ and $\mathfrak{A}_2$ are isomorphic to the direct sum of $C^{*}$-algebras with primitive ideal space less than $m$.  Hence, the result follows from our inductive hypothesis.  So, we may assume that $X$ is connected.  

We claim that for every $V \in \mathbb{O} ( X )$ with $V \neq \emptyset$, there exists an $(X \setminus V)$-equivariant isomorphism from $\mathfrak{A}_1[X \setminus V ]$ to $\mathfrak{A}_2 [ X \setminus V ]$.  Let $V \in \mathbb{O} (X)$ with $V \neq \emptyset$.  Set $Y = X \setminus V \in \mathbb{LC}(X)$.  Then $Z \in \mathbb{LC} ( Y)$ if and only if $Z \in \mathbb{LC} (X)$ and $Z \subseteq Y$.  Hence, $\alpha$ induces $\alpha^{Y} : \ouri_{Y, \Sigma} ( \mathfrak{A}_1[Y] ) \equiv \ouri_{Y, \Sigma}( \mathfrak{A}_2 [Y] )$.  Since $| Y | \leq m-1$, there exists a $Y$-equivariant isomorphism from $\mathfrak{A}_1[Y]$ to $\mathfrak{A}_2[Y]$.  We have just proved the claim.

Let $\mathfrak{I}^{1}_j, \mathfrak{I}^{2}_j, \dots, \mathfrak{I}^{n}_j$ be the minimal ideals of $\mathfrak{A}_j$.  Let $a_{i} \in X$ such that $\mathfrak{I}^{i}_j = \mathfrak{A}_j [ a_{i} ]$.  Set $\mathfrak{I}_j = \sum_{ i = 1}^{n} \mathfrak{I}^{i}_j$. Note that $U = \{ a_{1} , \dots, a_{n} \} \in \mathbb{O} ( X )$ such that $\mathfrak{I}_j = \mathfrak{A}_j [U]$ .  Since $X$ is connected, $\mathfrak{A}_1 [ a_{i} ]$ and $\mathfrak{A}_2 [ a_{i} ]$ are stable $C^{*}$-algebras and hence, $\mathfrak{A}_1 [ U ]$ and $\mathfrak{A}_2 [U]$ are stable $C^{*}$-algebras.  

Let for $j\in \{1,2\}$
\begin{align*}
\mathfrak{e}_{j} \ : \ 0 \to \mathfrak{I}_j  \to \mathfrak{A}_j \to \mathfrak{A}_j / \mathfrak{I}_j  \to 0 
\end{align*}  
Suppose $| U | = 1$.  Then $\mathfrak{e}_{1}$ and $\mathfrak{e}_{2}$ are full extensions.  Since $\mathfrak{A}_j / \mathfrak{I}_j$  are tight $C^{*}$-algebras over $X \setminus U$ and $| X \setminus U | = m - 1$, by the above claim, there exists an $X \setminus U$-equivariant isomorphism $\ftn{ \beta }{ \mathfrak{A}_1 / \mathfrak{I}_1 }{ \mathfrak{A}_2 / \mathfrak{I}_2 }$.  Also, there exists an isomorphism $\ftn{ \theta }{ \mathfrak{I}_1 }{ \mathfrak{I}_2 }$.  Since $\kk^{1} ( \mathfrak{A}_1 / \mathfrak{I}_1 , \mathfrak{I}_2 ) = 0$, we have that $\kk ( \beta ) \times [ \tau_{ \mathfrak{e}_{2} } ] = [ \tau_{ \mathfrak{e}_{1} } ] \times \kk ( \theta ) = 0$.  Since $\sigma_{ \mathfrak{A}_1 } = \sigma_{ \mathfrak{A}_2 }$, we have that $\mathfrak{A}_1$ and $\mathfrak{A}_2$ are both unital or both non-unital.  Hence, by Lemma \ref{lem:class}, there exists an isomorphism $\ftn{ \phi }{ \mathfrak{A}_1 }{ \mathfrak{A}_2 }$ such that $\pi_{2} \circ \phi = \beta \circ \pi_{1}$.  Since $\beta$ is an $X \setminus U$-equivariant isomorphism and $\theta$ is a $U$-equivariant isomorphism, by Lemma \ref{l:equivariant}, $\phi$ is an $X$-equivariant isomorphism.

Suppose $| U | \geq 2$.  Set $\hat{ \mathfrak{I} }^{k}_j = \sum_{ i = 1}^{k-1} \mathfrak{I}^{i}_j + \sum_{ i = j+1}^{n} \mathfrak{I}^{i}_j$.  Let $\ftn{ \pi_{1, k } }{ \mathfrak{I}_j  }{ \mathfrak{I}^{k}_j }$ be the natural projections.

Note that there exist extensions $\mathfrak{e}_{ j, k} \ : \ 0 \to \mathfrak{I}^{k}_j  \to \mathfrak{A}_j / \hat{\mathfrak{I}}^{k}_j  \to \mathfrak{A}_j / \mathfrak{I}_j \to 0$ such that 
\begin{equation*}
\xymatrix{
0 \ar[r] & \mathfrak{I}_1  \ar[r] \ar[d]^{ \pi_{j,k} } & \mathfrak{A}_j \ar[r] \ar[d] & \mathfrak{A}_j / \mathfrak{I}_j  \ar[r] \ar@{=}[d]& 0\\
0 \ar[r] & \mathfrak{I}^{k}_j  \ar[r] & \mathfrak{A}_j / \hat{ \mathfrak{A} }^{k}_j \ar[r] & \mathfrak{A}_j / \mathfrak{I}_j \ar[r] & 0 
}
\end{equation*}
By Theorem 2.2 of \cite{ELPmorph}, $\overline{ \pi }_{ j, k } \circ \tau_{ \mathfrak{e}_{j} } = \tau_{ \mathfrak{e}_{j, k } }$.  

We claim that there exists $U_{k} \in \mathbb{O} ( X )$ such that $U \subseteq U_{k}$ and $\mathfrak{A}_j [ U_{k} \setminus U ] = \ker ( \tau_{ \mathfrak{e}_{j, k } } )$.  Note that $\mathfrak{A}_j / \hat{ \mathfrak{I} }^{k}_j $ are tight $C^{*}$-algebras over $Y_{ k} = (X \setminus U) \cup \{ a_{k} \}$.  Moreover, $\mathfrak{A}_j / \hat{ \mathfrak{I} }^{k}_j$ are in $\mathcal{C}_{\mathrm{free}}$.

Set $\mathfrak{A}^{k}_j = \mathfrak{A}_j / \hat{ \mathfrak{I} }^{k}_j$. Since $| Y_{k} | \leq m-1$, there exists a $Y_{k}$-equivariant isomorphism $\ftn{ \psi }{ \mathfrak{A}^{k}_1  }{ \mathfrak{A}_2^{k}  }$.  Hence, by Theorem 2.2 of \cite{ELPmorph}, $\psi$ induces isomorphisms $\ftn{ \psi_{ X \setminus U } }{ \mathfrak{A}_1 / \mathfrak{I}_1 }{ \mathfrak{A}_2 / \mathfrak{I}_2  }$ and $\ftn{ \psi_{ \{ a_{k} \} } }{ \mathfrak{I}_{k} }{ \mathfrak{I}^{k}_2 }$ such that the diagram
\begin{equation*}
\xymatrix{
\mathfrak{A}_1 / \mathfrak{I}_1 \ar[r]^-{ \tau_{ \mathfrak{e}_{1,k} } } \ar[d]_-{\psi_{ X\setminus U} } & \corona{ \mathfrak{I}^{k}_1 } \ar[d]^-{ \overline{\psi}_{ \{ a_{k} \} } } \\
\mathfrak{A}_2 / \mathfrak{I}_2\ar[r]_-{ \tau_{ \mathfrak{e}_{2,k} } } & \corona{ \mathfrak{I}^{k}_2 }
}
\end{equation*}
is commutative.  Since the vertical maps are isomorphism, $\psi_{ X\setminus U } ( \ker ( \tau_{ \mathfrak{e}_{1, k } } ) ) = \ker ( \tau_{ \mathfrak{e}_{2, k } } )$.  Let $U_{k} \in \mathbb{O} ( X )$ such that $U \subseteq U_{k}$ and $\mathfrak{A}_1 [ U_{k} \setminus U ] = \ker ( \tau_{ \mathfrak{e}_{1, k } } )$.  Since $\psi$ is a $Y_{k}$-equivariant isomorphism,  $\psi_{ X \setminus U } ( \mathfrak{A} [ U_{k} \setminus U] ) = \mathfrak{A}_2 [ U_{k} \setminus U ]$.  Hence, $\ker(  \tau_{ \mathfrak{e}_{2, k } } ) = \mathfrak{A}_2 [ U_{k} \setminus U ]$.

Note that $\mathfrak{A}_1 / \mathfrak{I}_1$ and $\mathfrak{A}_2 / \mathfrak{I}_2$ are tight $C^{*}$-algebras over $X \setminus U$.  Moreover, $\mathfrak{A}_1 / \mathfrak{I}_1$ and $\mathfrak{A}_2 / \mathfrak{I}_2$ are in $\mathcal{C}_{\mathrm{free}}$.  Since $| X \setminus U | \leq m-1$, there exists an $X \setminus U$-equivariant isomorphism $\ftn{ \beta }{ \mathfrak{A}_1 / \mathfrak{I}_1 }{ \mathfrak{A}_2 / \mathfrak{I}_2  }$.  

Note that there exist injective homomorphisms $\ftn{ \overline{ \tau }_{ \mathfrak{e}_{j, k } } }{ ( \mathfrak{A}_j / \mathfrak{I}_j  ) / \ker( \tau_{ \mathfrak{e}_{j,k} } ) }{ \corona{ \mathfrak{I}^{k}_j }  }$ such that the diagrams
\begin{equation*}
\xymatrix{
\mathfrak{A}_j / \mathfrak{I}_j  \ar[r]^-{ \tau_{ \mathfrak{e}_{j,k } } } \ar[d] & \corona{ \mathfrak{I}^{k}_j  } \\
( \mathfrak{A}_j / \mathfrak{I}_j  ) / \ker( \tau_{ \mathfrak{e}_{j,k} } ) \ar[ru]_-{\overline{ \tau }_{ \mathfrak{e}_{j, k } }} & } 
\end{equation*}
are commutative.  

Since $\beta$ is an $X\setminus U$-equivariant isomorphism and since $\ker[ \tau_{ \mathfrak{e}_{j,k} } ] = \mathfrak{A}_j [ U_{k} \setminus U ]$, the map $\ftn{ \beta_{ X \setminus U_{k } } }{ ( \mathfrak{A}_1 / \mathfrak{I}_1  ) / \ker( \tau_{ \mathfrak{e}_{1,k} } ) }{ ( \mathfrak{A}_2 / \mathfrak{I}_2  ) / \ker( \tau_{ \mathfrak{e}_{2,k} } ) }$ is an isomorphism.  Note that there exists an isomorphism $\ftn{ \phi_{ k } }{ \mathfrak{I}^{k}_1  }{ \mathfrak{D}^{k}_2  }$.  Since $\kk^{1} ( ( \mathfrak{A}_1 / \mathfrak{I}_1  ) / \ker( \tau_{ \mathfrak{e}_{1,k} } ), \mathfrak{I}^{k}_2  ) = 0$, we have that 
\begin{equation*}
[ \overline{ \tau }_{ \mathfrak{e}_{2,k} } \circ \beta_{ X \setminus U_{k} } ] = [ \overline{ \phi }_{k} \circ \overline{ \tau }_{ \mathfrak{e}_{1,k} } ]. 
\end{equation*} 
Since $\overline{ \tau }_{ \mathfrak{e}_{2,k} } \circ \beta_{ X \setminus U_{k} }$ and $\overline{ \phi }_{k} \circ \overline{ \tau }_{ \mathfrak{e}_{1,k} }$ are essential extensions, they are full extensions since $\mathfrak{D}_{k} $ is isomorphic to $\K$ or a stable purely infinite simple $C^{*}$-algebra.

Since  $\mathfrak{A}_1$ and $\mathfrak{A}_2$ are in $\mathcal{C}_{\mathrm{free} }$ and since $\sigma_{ \mathfrak{A}_{1} } = \sigma_{ \mathfrak{A}_2 }$, we have that either $\overline{\tau}_{\mathfrak{e}_{1,k}}$ and $\overline{\tau}_{\mathfrak{e}_{2,k}}$ are both non-unital extensions or they are both unital extensions.  In the unital extension case,  $K_{0} ( ( \mathfrak{A} / \mathfrak{I}  ) / \ker( \tau_{ \mathfrak{e}_{1,k} } ) ) \cong \bigoplus_{ \mathcal{I} } \Z$ such that $[ 1_{ ( \mathfrak{A} / \mathfrak{I}  ) / \ker( \tau_{ \mathfrak{e}_{1,k} } ) } ]$ is mapped to $(1, x )$.  Hence, by Theorem 2.4 and Corollary 3.8 of \cite{HL_unitaryequiv} and Lemma \ref{l:grplin}, there exists a unitary $u_{k} \in \multialg{ \mathfrak{D}_{k}  }$ such that 
\begin{equation*}
\text{Ad} ( \overline{u}_{k} ) \circ \overline{ \phi }_{k} \circ \overline{ \tau }_{ \mathfrak{e}_{1,k} } = \overline{ \tau }_{ \mathfrak{e}_{2,k} } \circ \beta_{ X \setminus U_{k} }. 
\end{equation*} 
Since $\beta ( \ker( \tau_{ \mathfrak{e}_{1,k} } ) ) = \ker( \tau_{ \mathfrak{e}_{2,k} } )$ and since the diagram
\begin{equation*}
\xymatrix{
0 \ar[r] & \ker( \tau_{ \mathfrak{e}_{1,k} } ) \ar[r] \ar[d]^{ \beta } & \mathfrak{A} / \mathfrak{I}  \ar[r] \ar[d]^{ \beta } & ( \mathfrak{A} / \mathfrak{I}  ) / \ker( \tau_{ \mathfrak{e}_{1,k} } ) \ar[r] \ar[d]^{ \beta_{ X \setminus U_{k} } } & 0 \\
0 \ar[r] & \ker( \tau_{ \mathfrak{e}_{2,k} } ) \ar[r] & \mathfrak{A}_2 / \mathfrak{D}  \ar[r] & ( \mathfrak{A}_2 / \mathfrak{D}  ) / \ker( \tau_{ \mathfrak{e}_{2,k} } ) \ar[r] & 0
}
\end{equation*}
is commutative,
\begin{align*}
\text{Ad} ( \overline{u}_{k} ) \circ \overline{ \phi }_{k}  \circ \overline{ \pi }_{ 1, k } \circ \tau_{ \mathfrak{e}_{1} } &=\text{Ad} ( \overline{u}_{k} ) \circ \overline{ \phi }_{k} \circ  \tau_{ \mathfrak{e}_{1,k} } \\
					&= \tau_{ \mathfrak{e}_{2,k} } \circ \beta \\
					&= \overline{ \pi }_{ 2, k } \circ \tau_{ \mathfrak{e}_{2} }. \circ \beta
\end{align*} 

Define $\ftn{ \theta }{ \mathfrak{I}  }{ \mathfrak{D}  }$ by $\theta ( \sum_{ i = 1}^{n} x_{i} ) = \sum_{ i = 1}^{n} \text{Ad} ( u_{k} ) \circ \phi_{k} ( x_{k} )$.  Since $\mathfrak{I}_{k} \cap \mathfrak{I}_{\ell} = 0$ for $k \neq \ell$ and $\mathfrak{D}_{k} \cap \mathfrak{D}_{\ell} = 0$, $\theta$ is an $U$-equivariant isomorphism such that 
\begin{align*}
\overline{ \pi }_{2,k} \circ \overline{ \theta } \circ \tau_{ \mathfrak{e}_{1} } &= \text{Ad} ( \overline{u}_{k} ) \circ \overline{ \phi }_{k}  \circ \overline{ \pi }_{ 1, k } \circ \tau_{ \mathfrak{e}_{1} } \\	
				&= \overline{ \pi }_{ 2, k } \circ \tau_{ \mathfrak{e}_{2} } \circ \beta.
\end{align*}
Hence, $\overline{ \theta } \circ \tau_{ \mathfrak{e}_{1} } = \tau_{ \mathfrak{e}_{2} } \circ \beta$.  By Theorem 2.2 of \cite{ELPmorph}, there exists an isomorphism $\ftn{ \lambda }{ \mathfrak{A} }{ \mathfrak{A}_2 }$ such that 
\begin{equation*}
\xymatrix{
0 \ar[r] & \mathfrak{I}  \ar[r] \ar[d]^{ \theta } & \mathfrak{A}  \ar[r] \ar[d]^{ \lambda } & \mathfrak{A} / \mathfrak{I}  \ar[r] \ar[d]^{ \beta } & 0 \\
0 \ar[r] & \mathfrak{D}  \ar[r] & \mathfrak{A}_2  \ar[r] & \mathfrak{A}_2 / \mathfrak{D}  \ar[r] & 0
}
\end{equation*}
By Lemma \ref{l:equivariant}, $\lambda$ is an $X$-equivariant isomorphism.
\end{proof}

\begin{corol}\label{c:classfreekthy}
Let $\mathfrak{A}$ and $\mathfrak{B}$ be in $\mathcal{C}_{ \mathrm{free} }$ with
finitely many ideals.  Then 
\begin{itemize}
\item[(i)] $\mathfrak{A} \cong \mathfrak{B}$ if and only if $\primTS ( \mathfrak{A}
) \cong \primTS ( \mathfrak{B} )$.

\item[(ii)] $\mathfrak{A} \otimes \K \cong \mathfrak{B} \otimes \K$ if and only if
$\primT ( \mathfrak{A} ) \cong \primT ( \mathfrak{B} )$.
\end{itemize}
\end{corol}

\begin{proof}
Set $X = \Prim ( \mathfrak{A} )$ and $Y = \Prim ( \mathfrak{B} )$.  We first prove
(i).  Suppose $\primTS ( \mathfrak{A} ) \cong \primTS ( \mathfrak{B} )$, then there
exists a homeomorphism $\ftn{ \alpha }{ X }{ Y }$ such that 
\begin{align*}
\tau_{ \mathfrak{A} } = \tau_{ \mathfrak{B} } \circ \alpha \quad \text{and} \quad
\sigma_{ \mathfrak{A} } = \sigma_{ \mathfrak{B} } \circ \widetilde{\alpha}.
\end{align*}
Define a $C^{*}$-algebra $\mathfrak{C}$ over $X$ by $\mathfrak{C} [ Y ]  =
\mathfrak{B} [ \alpha(Y)]$ for each $Y \in \mathbb{LC}(X)$.  Then $\mathfrak{C}$ is
a tight $C^{*}$-algebra over $X$.  By construction, $\mathfrak{C} \cong
\mathfrak{B}$ and $\ouri_{X, \Sigma } ( \mathfrak{C} )  \equiv \ouri_{X, \Sigma } (
\mathfrak{A} )$.  Therefore, by Theorem \ref{t:classification}, $\mathfrak{C} \cong
\mathfrak{A}$.  Hence, $\mathfrak{A} \cong \mathfrak{B}$. 

(ii) follows from (i) since $\sigma_{ \mathfrak{A} } = 0 = \sigma_{ \mathfrak{B} }$.  
\end{proof}

\begin{corol} \label{c:singularisamp}
Let $G$ be a singular graph with finitely many vertices and no breaking vertices. 
Then $C^*(G) \cong C^*(\overline{G})$.
\end{corol}
\begin{proof}
By Proposition \ref{p:apgraphclass} $C^*(G)$ is in $\mathcal{C}_{\mathrm{free}}$. 
One can easily check that since $G$ has no breaking vertices the ideal spaces of $C^*(G)$ and $C^*(\tc{G})$ are the same. 
Likewise for the $K$-groups. 
Therefore $\primTS (C^*(G)) \cong \primTS (C^*(\overline{G}))$ and so $C^*(G) \cong C^*(\overline{G}) \cong C^*(\tc{G})$. 
\end{proof}

It turns out that if we restrict our category to unital $C^{*}$-algebras in $\mathcal{C}_{\mathrm{free}}$, then $\ouri ( \cdot )$ is a classification functor.  To do this we need the following lemma.

\begin{lemma}\label{l:unitalsubquot}
Let $\mathfrak{A}$ and $\mathfrak{B}$ be unital $C^{*}$-algebras such that $\Prim (
\mathfrak{A} )$ and $\Prim( \mathfrak{B} )$ are finite.  Suppose there exists a
homeomorphism $\ftn{ \alpha }{ \Prim ( \mathfrak{A} ) }{ \Prim ( \mathfrak{B} ) }$. 
Then for each $Y \in \mathbb{LC} ( \Prim ( \mathfrak{A} ) )$, $\mathfrak{A} [ Y ]$
is unital if and only if $\mathfrak{B} [ \widetilde{\alpha}(Y) ]$ is unital.  
\end{lemma}

\begin{proof}
Set $X = \Prim ( \mathfrak{A} )$ and $Y = \Prim ( \mathfrak{B} )$.  Let $U, V \in
\mathbb{O} ( X )$ such that $U \subseteq V$ and $\ftn{ \beta }{ X }{ Y }$ be the
homeomorphism given by $\alpha$.  Set $W = V \setminus U$, $Z = X \setminus U$, $S =
X \setminus V$.

Suppose $\mathfrak{A} [ W ]$ is unital.    Since $\mathfrak{A} [W]$ is unital, $W$
and $S$ are both open and closed subsets of $Z$.  Moreover, $Z$ is homeomorphic to
$W \sqcup S$.  Since $\alpha$ is a homeomorphism, $\alpha(W)$ and $\alpha(S)$ are
both open and closed subsets of $\alpha(Z) = Y \setminus \alpha(U)$ and $\alpha(Z)$
is homeomorphic to $\alpha(W) \sqcup \alpha(S)$.  Thus, $\mathfrak{A} [ W ] =
\mathfrak{A} [ Z \setminus S ]$ and $\mathfrak{B} [ \alpha(W) ] = \mathfrak{B} [
\alpha(Z) \setminus \alpha(S) ]$.  Since $\mathfrak{B} [ \alpha(Z) ]$ is unital,
$\mathfrak{B} [ \alpha(Z) \setminus \alpha(S) ]$ is unital.  Hence, $\mathfrak{B} [
\alpha(W) ]$ is unital.  

A similar argument shows that if $\mathfrak{B} [ \beta(W) ]$ is unital, then
$\mathfrak{A} [ W ]$ is unital.  
\end{proof}

\begin{theor}\label{t:unitalclass}
Let $\mathfrak{A}$ and $\mathfrak{B}$ be unital $C^{*}$-algebras in $\mathcal{C}_{\mathrm{free}}$, with finitely many ideals.  
Suppose $\ouri ( \mathfrak{A} ) \cong \ouri( \mathfrak{B} )$.
Then $\mathfrak{A} \cong \mathfrak{B}$.
\end{theor}

\begin{proof}
Let $\ftn{ \alpha }{ \Prim ( \mathfrak{A} ) }{ \Prim ( \mathfrak{B} ) }$ be a
homeomorphism such that 
\begin{align*}
\tau_{ \mathfrak{A} } = \tau_{ \mathfrak{B} } \circ \alpha.
\end{align*}
Let $x \in \Prim ( \mathfrak{A} )$.  By Lemma \ref{l:unitalsubquot}, we have that
$\mathfrak{A} [ Y ]$ is unital if and only if $\mathfrak{B} [ \widetilde{\alpha}(Y)
]$ is unital for all $Y \in \mathbb{LC} ( \Prim ( \mathfrak{A} ) )$.  Therefore, 
\begin{align*}
\tau_{ \mathfrak{A} } = \tau_{ \mathfrak{B} } \circ \widetilde{ \alpha }.
\end{align*}  
Thus, we have that $\primTS ( \mathfrak{A} ) \cong \primTS ( \mathfrak{B} )$.  By
Corollary \ref{c:classfreekthy}, $\mathfrak{A} \cong \mathfrak{B}$.  
\end{proof}

\section{Range of the invariant and permanence properties}

We saw in Section \ref{classamplified} that $\primT( \cdot )$ is a classification functor for the class of graph $C^{*}$-algebras associated to amplified graphs with finitely many vertices.  In fact, we showed in Proposition \ref{p:strongclassamp} that $\primT( \cdot )$ is a strong classification functor.  We now determine the range of $\primT( \cdot )$.  

Let $G$ be a finite graph.  By Proposition \ref{p:apgraphclass}, $C^{*} ( \overline{G} ) \in \mathcal{C}_{ \mathrm{free} }$.  Hence, $X = \mathrm{Prim} ( C^{*} ( \overline{G} ) )$ is finite and for each $x \in X$, $\tau_{ C^{*} ( \overline{G} ) } ( x ) \in \N$ when $K_{0} ( C^{*} ( \overline{G} ) [ x ] )_{+} = K_{0} ( C^{*} ( \overline{G} ) )$ and $\tau_{ C^{*} ( \overline{G} ) } ( x )  = -1$ when $K_{0} ( C^{*} ( \overline{G} ) [ x ] )_{+} \neq K_{0} ( C^{*} ( \overline{G} ) )$.  We will show in this section that this is the only obstruction for the range of $\primT( \cdot )$.

\begin{lemma}\label{l:primideal}
Let $(X , \preceq )$ be a finite partially ordered set and let $F$ be the acyclic graph representing $(X , \preceq )$ described in section \ref{sec:alexandrov}.  Let $F^{op}$ be the graph obtained from $F$ by reversing the arrows of $F$.  Then $\Prim ( C^{*} ( \overline{ F^{op} } ) ) \cong X$.
\end{lemma}

\begin{proof}
It is clear that $H$ is a hereditary subset of $\overline{ F^{op} }^{0}$ if and only if $H$ is open in $X$.  Since $\overline{ F^{op} }$ is a singular graph with no breaking vertices, we have $\text{Prim} ( C^{*} ( \overline{ F^{op} } ) ) \cong X$.
\end{proof}

\begin{lemma}\label{l:ran}
Let $G$ be a finitely generated free abelian group.  Then there exists a strongly connected finite graph $E$ such that $C^{*} ( \overline{E} )$ is a purely infinite simple $C^{*}$-algebra, $| \overline{E}^{0} | = \text{rank} ( G )$, and 
\begin{align*}
( K_{0} ( C^{*} ( \overline{E} ) )  , K_{0} ( C^{*} ( \overline{E} ) )_{+} ) \cong (G, G )
\end{align*}
\end{lemma}

\begin{proof}
Set $m = \text{rank} ( G )$.  Define $E$ by $E^{0} = \{ v_{1}, v_{2}, \dots, v_{m} \}$,
\begin{align*}
E^{1} = \setof{ e(v, w ) }{v, w \in \{ v_{1}, v_{2}, \dots, v_{m} \} },
\end{align*}
$s_{F} ( e(v,w) ) = v$ and $r_{F} ( e(v,w)) = w$.  It is clear that $E$ is a strongly connected.  Since $E$ contains a cycle, $C^{*} ( \overline{E} )$ is purely infinite and by Corollary 3.2 of \cite{ddmt_kthygraph} and Theorem 2.2 of \cite{mt:orderkthy}, $K_{0} ( C^{*} ( \overline{E} ) ) \cong \bigoplus_{ v \in \overline{E} ^{0} } \Z \cong G$.  Since $C^{*} ( \overline{E} )$ is purely infinite, 
\[
( K_{0} ( C^{*} ( \overline{E} ) )  , K_{0} ( C^{*} ( \overline{E} ) )_{+} ) \cong (G, G ). \qedhere
\]
\end{proof}

\begin{theor}\label{t:range}
Let $X$ be a finite topological space and let $\ftn{f}{ X }{ \{ -1 \} \cup  \N }$ be a function.  Then there exist a finite graph $G$ and a homeomorphism $\ftn{ \alpha }{ \Prim ( C^{*} ( \overline{G} ) ) }{ X }$ such that 
\begin{align*}
f \circ \alpha = \tau_{ C^{*} ( \overline{G} ) }.
\end{align*}
\end{theor}

\begin{proof}
By Lemma \ref{l:primideal}, there exists a finite graph $H_{0}$ such that  $\text{Prim} ( C^{*} ( \overline{ H_{0} } ) ) \cong X$ and $\overline{H_{0}}^{0} = X$. Set $H = \overline{H_{0}}$.  Let $v$ be $H^{0}$.  If $f( v ) > 0$, then by Lemma \ref{l:ran}, there exists a strongly connected singular graph $E_{v}$ with finitely many vertices and no breaking vertices such that $\mathrm{rank} ( K_{0} ( C^{*} ( E_{v} ) ) ) = f(v)$ and $C^{*} ( E_{v} )$ is a purely infinite simple $C^{*}$-algebra.  If $f(v) < 0$, set $E_{v} = \{ v \}$.

For each $v \in H^{0}$, let $w_{v}$ be an element of $E_{v}^{0}$.  Define $E$ as follows $E^{0} = \bigcup_{ v \in H^{0} } E_{v}^{0}$, 
\begin{align*}
E^{1} = \left( \bigcup_{ v \in H^{0} } E_{v}^{1} \right) \cup \setof{ e( w_{v}, w_{z} )^{n} }{ \text{$n \in \N$, $\exists$ edge in $H$ from $v$ to $w$} }
\end{align*}
$s_{E} \vert_{ E_{v} } = s_{E_{v}}$, $s_{E} (  e( w_{v}, w_{z} )^{n} ) = w_{v}$, $r_{E} \vert_{ E_{v} } = r_{ E_{v}}$, and $r_{E} ( e( w_{v}, w_{z} )^{n} ) = w_{z}$.  Then $E \cong \overline{G}$ for some finite graph $G$.  

Define $\ftn{ \beta }{ \mathbb{O} ( X ) }{  \Her ( E ) }$ by $\beta ( U ) = \cup_{ v \in U }E^{0}_{v}$.  By the construction of $E$, we have that $\beta$ is a lattice isomorphism.  Thus, $\beta$ induces a homeomorphism $\ftn{ \widetilde{\beta} }{ X }{ \Prim ( C^{*} (E) ) }$ such that $C^{*} ( E ) [ \widetilde{\beta} (v)  ] \cong C^{*} ( E_{v} )$.  Set $\alpha = \widetilde{\beta}^{-1}$.  Then $\alpha$ is a homeomorphism.  Let $x \in \Prim ( C^{*} ( E ) )$ and let $v = \alpha ( x )$.  Thus, 
\begin{align*}
C^{*} (E) [ x ] = C^{*} ( E ) [ \widetilde{ \beta } ( v )  ] \cong C^{*} ( E_{v} ).
\end{align*}
Hence, $f \circ \alpha = \tau_{ C^{*} ( E ) }$
\end{proof}

As a consequence of the above theorem and our general classification result (Proposition \ref{p:apgraphclass} and Theorem \ref{t:unitalclass}), we have that every unital $C^{*}$-algebra in $\mathcal{C}_{ \mathrm{free} }$ with finitely generated $K$-theory is isomorphic to $C^{*} ( \overline{G} )$ for some finite graph $G$. 

\begin{corol}\label{c:unitalfree}
Let $\mathfrak{A}$ be a unital $C^{*}$-algebra in $\mathcal{C}_{\mathrm{free}}$ with $K_{0} ( \mathfrak{A} )$ finitely generated.  Then there exists a finite graph $G$ such that $\mathfrak{A} \cong C^{*} ( \overline{G} )$.  
\end{corol}

\begin{proof}
Note that $\Prim ( \mathfrak{A} )$ is finite.  Since $\mathfrak{A}$ is finitely generated and 
\[
	K_{0} ( \mathfrak{A} ) \cong \bigoplus_{ x \in \Prim ( \mathfrak{A} ) } K_{0} ( \mathfrak{A} [ x ] ),
\]
$K_{0} ( \mathfrak{A} [x ] )$ is finitely generated for all $x \in \Prim ( \mathfrak{A} )$.   Thus $\tau_{ \mathfrak{A} } ( x ) \in \{ -1 \} \cup \N$.  By Theorem \ref{t:range}, there exists a finite graph $G$ such that $\ouri ( \mathfrak{A} ) \cong \ouri ( C^{*} ( \overline{G} ) )$.  Hence, by Proposition \ref{p:apgraphclass} and Theorem \ref{t:unitalclass}, $\mathfrak{A} \cong C^{*} ( \overline{G} )$. 
\end{proof}

Using our general classification result and our range result, we can achieve a permanence result for extensions of graph algebras associated to amplified graphs.

\begin{corol}\label{c:extensions}
Let $G_{1}$ and $G_{2}$ be a finite graphs.  If $\mathfrak{A}$ is a unital $C^{*}$-algebra and $\mathfrak{A}$ fits into the following exact sequence
\begin{align*}
0 \to C^{*}( \overline{G_{1}} ) \otimes \K \to \mathfrak{A} \to C^{*} ( \overline{G_{2}} ) \to 0
\end{align*}
then $\mathfrak{A} \in \mathfrak{C}_{ \mathrm{free} }$.  Consequently, there exists a finite graph $G$ such that $\mathfrak{A} \cong C^{*} ( \overline{G} )$.
\end{corol}

\begin{proof}
Note that by Proposition \ref{p:apgraphclass} and Proposition \ref{p:elements}, $C^{*} ( \overline{G_{1}} ) \otimes \K$ and $C^{*} ( \overline{G_{2}} )$ are elements of $\mathcal{C}_{ \mathrm{free} }$.  

Set $X = \Prim ( \mathfrak{A} )$.  Let $U$ be an open subset of $X$ such that $\mathfrak{A} [ U ] \cong C^{*}( \overline{G_{1}} ) \otimes \K$ and $\mathfrak{A} [ X \setminus U ] \cong C^{*} ( \overline{G_{2}} )$.  Set $Y = X \setminus U$.  Since $\Prim ( C^{*} ( G_{i} ) )$ is finite, $U$ and $Y$ are finite.  Hence, $X$ is finite.

Since $\mathfrak{A} [ U ]$ is a tight $C^{*}$-algebra over $U$ and $\mathfrak{A} [ Y ]$ is a tight $C^{*}$-algebra over $X$, there exist homeomorphisms $\ftn{ \beta_{U} }{ \Prim ( C^{*} ( G_{1} ) \otimes \K ) }{ U }$ and $\ftn{ \beta_{Y} }{ \Prim ( C^{*}( G_{2} ) ) }{ Y }$. Let $x \in \Prim ( \mathfrak{A} )$.  Then $x \in U$ or $x \in X \setminus U$.  If $x \in U$, then $\mathfrak{A} [ x ] \cong ( C^{*} ( G_{1} ) \otimes \K) [ \beta_{U}^{-1}(x) ]$ and if $x \in X \setminus U$, then $\mathfrak{A} [ x ] \cong C^{*} ( G_{1} ) [ \beta_{Y}^{-1}(x) ]$.    Hence, by Proposition \ref{p:elements}, $\mathfrak{A} [ x ]$ are elements of $\mathcal{C}_{ \mathrm{free} }$.  Thus, by Proposition \ref{p:elements}, $\mathfrak{A} \in \mathcal{C}_{ \mathrm{free} }$.  

The last part of the statement follows from Corollary \ref{c:unitalfree} since $K_{0} ( \mathfrak{A} ) \cong K_{0} ( C^{*} ( \overline{G_{1} } ) ) \oplus K_{0} ( C^{*} ( \overline{G_{2}} ) )$ which implies $K_{0} ( \mathfrak{A} )$ is finitely generated.
\end{proof}

\section{Acknowledgement}

The second author wishes to thank the Department of Mathematical Sciences at the University of Copenhagen for their support and hospitality during his visit in the spring of 2010, during which some of this research were initiated.  

This research was supported by the Danish National Research Foundation (DNRF) through the
Centre for Symmetry and Deformation. Support was also provided by the NordForsk Research
Network ``Operator Algebras and Dynamics'' (grant \#11580).


\end{document}